\documentclass[11pt]{amsart}
\usepackage{amssymb, amscd}
\usepackage{latexsym,epsfig}
\usepackage[all]{xy}
\usepackage{pst-all}
\numberwithin{equation}{section}
\def\today{\ifcase\month\or Jan\or Febr\or  Mar\or  Apr\or May\or Jun\or  Jul\or Aug\or  Sep\or  Oct\or Nov\or  Dec\or\fi \space\number\day, \number\year}

\newcommand{\CC}{\mathbb C}

\newcommand{\GG}{\mathbb G}

\newcommand{\PP}{\mathbb P}
\newcommand{\QQ}{\mathbb Q}
\newcommand{\RR}{\mathbb R}

\newcommand{\ZZ}{\mathbb Z}

\newtheorem{theorem}{Theorem}[section]
\newtheorem{lemma}[theorem]{Lemma}
\newtheorem{proposition}[theorem]{Proposition}
\newtheorem{corollary}[theorem]{Corollary}

\newtheorem{definition-lemma}[theorem]{Definition-Lemma}

\theoremstyle{definition}
\newtheorem{definition}[theorem]{Definition}
\newtheorem{example}[theorem]{Example}

\theoremstyle{remark}
\newtheorem{remark}[theorem]{Remark}

\begin{document}

\title[]{Generators for Modules of Vector-valued \\
Picard Modular Forms }
\author{Fabien Cl\'ery}
\address{Korteweg-de Vries Instituut, Universiteit van
Amsterdam, Postbus 94248,
1090 GE  Amsterdam, The Netherlands.}
\email{f.l.d.clery@uva.nl}

\author{Gerard van der Geer}
\address{Korteweg-de Vries Instituut, Universiteit van
Amsterdam, Postbus 94248,
1090 GE  Amsterdam, The Netherlands.}
\email{geer@science.uva.nl}

\subjclass{14J15, 10D}
\begin{abstract}
We construct generators for modules of vector-valued Picard modular
forms on a unitary group of type (2,1) over the Eisenstein integers.
We also calculate eigenvalues of Hecke operators acting on cusp forms.
\end{abstract}

\maketitle
\begin{section}{Introduction}
Modular forms on unitary groups have been studied intensively
in the theory of automorphic forms, (cf.\ for example \cite{Rog} and
\cite{LR})
but explicit examples have been scarce. 
Shintani considered vector-valued Picard modular forms in an unpublished 
manuscript \cite{Shintani}; in particular he determined a criterion for such a
modular form to be a Hecke eigenform in terms of the Fourier-Jacobi series.
Explicit (scalar-valued) Picard modular forms were considered 
for $F={\QQ}(\sqrt{-1})$
by Resnikoff and Tai (cf.\ \cite{R-T1,R-T2}) and for
$F={\QQ}(\sqrt{-3})$ by Shiga, Holzapfel, Feustel and Finis, cf.\
\cite{Shiga, Holzapfel1, Feustel, Finis1}; in particular for the latter case
Holzapfel and
Feustel determined a presentation of the ring of
scalar-valued Picard modular forms on the congruence subgroup
$$
\Gamma_1[\sqrt{-3}]=\{ g \in \Gamma_1 : g \equiv 1 (\bmod \sqrt{-3})\},
$$
while Finis computed a number of Hecke eigenvalues for low weight ($\leq 12$)
scalar-valued forms. Vector-valued Picard modular forms have not attracted much
attention so far.

There is another approach to calculating Hecke eigenvalues of modular forms
via the
study of the cohomology of local systems
on moduli spaces of algebraic curves. This approach which uses counting of
points over finite fields, was pioneered 
by Faber and the second author in 
\cite{FvdG} and continued with Bergstr\"om in \cite{BFvdG}. It has provided
a lot of explicit examples and gathered information that led to 
conjectures on vector-valued Siegel modular forms of genus $2$ and $3$. 

The Picard modular surface underlying the work of Feustel, Holzapfel and
Finis can be interpreted as a moduli space of degree three Galois 
covers of genus $3$ of the projective line, cf.\  early work
of Shimura \cite{Shimura1}. Using this interpretation Bergstr\"om and
the second author set out to calculate traces of Hecke operators on spaces
of vector-valued modular forms by calculating the trace of Frobenius
on the \'etale cohomology of local systems over finite fields. This
approach is carried out in \cite{B-vdG}. 
It naturally led to the question of 
constructing the corresponding Picard modular forms directly. 
That is the starting point of the present paper.

The purpose of the paper is to construct vector-valued (eigen)forms on the
Picard modular group in question, that is, the congruence subgroup 
$\Gamma_1[\sqrt{-3}]$. The weight of a modular form is a
pair $(j,k)$, with the case $(0,k)$ corresponding to scalar-valued
modular forms of weight $k$. We will denote by
 $M_{j,k}(\Gamma[\sqrt{-3}],{\rm det}^{\ell})$ 
(resp.\ $S_{j,k}(\Gamma[\sqrt{-3}],{\rm det}^{\ell})$ 
the vector space of modular forms (resp.\ cusp forms) of weight $(j,k)$
with character ${\rm det}^{\ell}$ with $0 \leq \ell \leq 2$ 
(see Section \ref{sec:defs} for precise definitions). 
Then 
$${\mathcal M}_j^{\ell}= \oplus_k M_{j,k}(\Gamma[\sqrt{-3}],{\rm det}^{\ell})
\quad{\rm and} \quad
\Sigma_j^{\ell}=\oplus_k S_{j,k}(\Gamma[\sqrt{-3}],{\rm det}^{\ell})
$$ 
are modules over the ring $M={\mathcal M}_0^0= \oplus M_{0,k}(\Gamma[\sqrt{-3}])$ 
of scalar-valued modular forms. 
Generators for the ring $M$ of scalar-valued modular forms are known 
by work of Feustel and Holzapfel, in fact $M={\CC}[\varphi_0,\varphi_1,\varphi_2]$ with $\varphi_i$ of weight $3$. 
We give a presentation for the $M$-modules ${\mathcal M}_j^{\ell}$ 
and $\Sigma_j^{\ell}$ 
for $j=1$ and $2$. We also give the structure for ${\mathcal M}_3^0$.
To illustrate this we give one example.

\begin{theorem}
The $M$-module ${\mathcal M}_1^0$ is generated by three vector-valued
cusp forms $\Phi_0$, $\Phi_1$ and $\Phi_2$ of weight $(1,7)$ 
satisfying a single relation
$\varphi_0\Phi_0+\varphi_1\Phi_1+\varphi_2\Phi_2=0$. 
\end{theorem}

We can calculate the Fourier-Jacobi expansions of our forms and 
thus can calculate the Hecke eigenvalues of these modular forms. We give
Hecke eigenvalues for a number of generators. The corresponding Galois 
representations are of dimension $1$, $2$ or $3$.
These results agree with the cohomological results of 
Bergstr\"om and van der Geer
\cite{B-vdG}. The results of \cite{B-vdG} on the $S_4$-equivariant
numerical Euler characteristics of the local systems 
predicted where to look for generators of these modules
of modular forms. 
We hope these results will make Picard modular 
forms more tangible than they have been so far
and that the presence of such explicit examples can help discover new phenomena.

The case of Picard modular forms over the Gaussian integers
will be treated in a forthcoming publication;  also we intend to treat
cases of modular forms  on higher rank unitary groups.

\smallskip
\noindent
{\bf Acknowledgement} The results of this paper would not have been possible
without the guidance from the cohomological results 
of Jonas Bergstr\"om and the second author. 
Unfortunately, the preprint \cite{B-vdG} is not yet ready. We thank  Prof.\ H.\ Shiga and A.\ Murase for making Shintani's 
unpublished manuscript available for us.  
We also thank T.\ Finis for sending us his manuscripts on the subject.
The authors thank Jonas Bergstr\"om for some useful remarks.
This work was supported by a grant of NWO.
\end{section}
\begin{section}{The Picard Modular Group}\label{sec:defs}
Let $F$ be an imaginary quadratic field of discriminant $D$ 
with ring of integers $O_F$.
We consider a $3$-dimensional $F$-vector space $V$ that contains an
$O_F$-lattice $L$ with a non-degenerate hermitian form $h$ that is
$O_F$-valued on $L$ and of signature $(2,1)$. This determines an
algebraic group $G$ of unitary similitudes
$$
\{ g \in {\rm GL}(3,F): h(gz,gu)= \eta(g)h(z,u) \}
$$
defined over ${\QQ}$. We have $\eta(g)^3=
N_{F/{\QQ}}(\det(g)) \in {\QQ}_{>0}$ and
$\eta$ defines a homomorphism $\eta: G \to {\GG}_m$, called the multiplyer.  
The kernel $G^0:=\ker{\eta}$ is the usual unitary group and 
$G^0\cap \ker{\det}$ is the special unitary group
of signature $(2,1)$. The base change of $G$ to $F$ is isomorphic to
${\rm GL}(3,F)\times {\GG}_m$, where the latter factor corresponds to $\eta$.

We choose an embedding $\sigma: F \to {\CC}$ and identify $F\otimes_{\QQ} {\RR}$
with ${\CC}$. Then $V'=V\otimes_{\sigma}{\RR}$ becomes 
a $3$-dimensional complex
vector space and we can look at
$$
B:=\{ U \subset V' : \dim(U)=1, h|U < 0\}
\subset {\PP}(V')={\PP}^2,
$$
the set of complex lines on which $h$ is negative definite. The group 
$G^{+}=\{ g \in G({\RR}): \det(g)>0\}$
acts on the Grassmannian ${\rm Gr}(1,V')$ and on $B$. We can identify $B$ 
with the complex $2$-ball in ${\PP}^2$. 

A standard choice for the hermitian form is
$$
h(z,z)= \frac{1}{\sqrt{D}} z_1\bar{z}_3+z_2\bar{z}_2
-\frac{1}{\sqrt{D}} z_3\bar{z}_1
$$
on the lattice $O_F^3 \subset {\CC}^3$. This is a maximal lattice in the
sense of Shimura. Note that $z_1$ and $z_3$ do not vanish in $B$ and
by taking $u=z_1/z_3$ and $v=z_2/z_3$, the set of 
negative complex lines can be identified with the Siegel domain
$$
\{ (u,v) \in {\CC}^2: \frac{2}{\sqrt{|D|}}  {\rm Im}(u) +|v|^2 < 0 \}
$$
embedded in ${\PP}^2({\CC})$ via $(u,v) \mapsto (u:v:1)$.
If we identify $G({\QQ})$ with the matrix group
$$
\{ g \in {\rm GL}(3,F): \bar{g}^t H g = \eta(g) H \}
$$
with $H$ the skew-hermitian matrix
$$
H= \left( 
\begin{matrix} && 1/\sqrt{D} \\ &1 \\ -1/\sqrt{D} \\ 
\end{matrix} \right) \, ,
$$
then the action of $g=(g_{ij}) \in G$ on $B$ is given by
$$
(u,v)\mapsto 
\left(\frac{g_{11}u+g_{12}v+g_{13}}{g_{31}u+g_{32}v+g_{33}},
\frac{g_{21}u+g_{22}v+g_{23}}{g_{31}u+g_{32}v+g_{33}} \right)\, .
$$
The denominator $j_1(g,u,v):=g_{31}u+g_{32}v+g_{33}$ 
defines a factor of automorphy for this action.
The Jacobian $J(g,u,v)$ of the action of $G$ on $B$ defines a second factor
of automorphy:
$$
J(g,u,v)= j_1(g,u,v)^{-2} \left(
\begin{matrix} G_{23}v+G_{22} & G_{13}v+G_{12} \\
-G_{23}u+G_{21} & -G_{13}u+G_{11}\\
 \end{matrix} \right),
$$
where $G_{ij}$ is the minor of $g_{ij}$. One checks that 
$$
\det J(g,u,v)=
j_1(g,u,v)^{-3} \det(g) \eqno(1).
$$
We thus have two factors of automorphy
$$
j_1(g,u,v):=g_{31}u+g_{32}v+g_{33}
$$
and
$$
j_2(g,u,v):= \det(g)^{-1} \, \left( \begin{matrix}
-G_{13} u+G_{11} & -G_{13}v-G_{12} \\
G_{23}u-G_{21} & G_{23} v+G_{22} \\
\end{matrix}
\right) \, .
$$
Note that $\det(j_2(g,u,v))=j_1(g,u,v)/\det(g)$.
See also \cite{Shimura2} for the general case.

\bigskip
{\sl Our normalization.} 
Some authors use different normalizations. Via the coordinate change
$w_1=z_1/\sqrt{D},w_2=z_3,w_3=z_2$ we get the hermitian form
$$
w_1\bar{w}_2+w_2\bar{w}_1+w_3\bar{w}_3
$$ 
used for example by Finis (\cite{Finis1}).
The symmetric domain is given by
$$
B=\{ (u,v) \in {\CC}^2: 2 \, {\rm Re}\,  v+ |u|^2 < 0 \}
$$
with $u=w_3/w_2$, $v=w_1/w_2$ and the
action by
$$
(u,v) \mapsto \left(\frac{g_{31}v+g_{32}+g_{33}u}{g_{21}v+g_{22}+g_{23}u},
\frac{g_{11}v+g_{12}+g_{13}u}{g_{21}v+g_{22}+g_{23}u}\right).
$$
The factors of automorphy are then
$$
j_1(g,u,v)= g_{21}v+g_{22}+g_{23}u
$$
and
$$
j_2(g,u,v):=\det(g)^{-1} \left( \begin{matrix}
G_{32} u+G_{33} & G_{32}v+G_{31} \\
G_{12}u+G_{13} & G_{12} v+G_{11} \\
\end{matrix} \right) \eqno(2)
$$
so that we have
$$
j_2(g,u,v)^{-1}= j_1(g,u,v) \,  (J(g,u,v))^t \, . \eqno(3)
$$
Since we are using some of Finis' calculations  
we shall use this normalization in the sequel.

\bigskip
For a pair $(j,k)$ of integers and $g \in G({\RR})$ 
we define a slash operator on functions $f : B \to {\rm Sym}^j({\CC}^2)$ via
$$
(f|_{j,k} g)(u,v) = j_1(g,u,v)^{-k} {\rm Sym}^j(j_2(g,u,v)^{-1}) 
\, f(g\cdot (u,v))
$$
For a discrete subgroup $\Gamma$ of $G({\RR})$ and a character $\chi$ 
of $\Gamma$ of finite order 
we define the space of modular forms of weight $(j,k)$ and character $\chi$ as
$$
\begin{aligned}
M_{j,k}(\Gamma,\chi):= 
\{ f: B \to {\rm Sym}^j({\CC}^2): & \text{ $f$ holomorphic}, \\
& f|_{j,k}\,  g= \chi(g) \, f  
\text{ for all $g \in \Gamma$ } \} \\
\end{aligned}
$$
It contains a subspace $S_{j,k}(\Gamma,\chi)$ of cusp forms.

We will consider in particular the Picard modular group and the
special Picard modular group and fix $\Gamma$ as
$$
\Gamma= G^0({\ZZ}) \text{ and } \Gamma_1=G^0({\ZZ}) \cap \ker \det .
$$
The quotient group $\Gamma/\Gamma_1$ is isomorphic to the roots of
unity contained in $O_F$. 
Note that $N_{F/{\QQ}}(\det(g))$ is a positive integer and a unit, 
hence $\det(g)$ is a root of unity in $O_F$. As characters $\chi$ we will
consider only powers of $\det (g)$. If $\chi={\rm id}$ then we suppress
the notation $\chi$ and write $M_{j,k}(\Gamma)$ instead.

We thus have the notions of vector-valued 
Picard modular forms with character on the groups $\Gamma$ and $\Gamma_1$.
We can consider the ring of scalar-valued modular forms
$$
{\mathcal M}(\Gamma)= \oplus_k M_{0,k}(\Gamma)
$$
and for fixed $j\geq 0$ the ${\mathcal M}(\Gamma)$-module 
$$
{\mathcal M}_{j}(\Gamma):=\oplus_{k} M_{j,k}(\Gamma)
$$ 
and similarly for other discrete groups.

\end{section}
\begin{section}{The Picard Modular Group for $F={\QQ}(\sqrt{-3})$}\label{ourcase}
We now specialize to the case where $F={\QQ}(\sqrt{-3})={\QQ}(\rho)$ 
with $\rho$ a third root of unity and where $V=F^3$ with hermitian form given by
$$
\left(\begin{matrix} 0 & 1 & 0 \\ 1 & 0 & 0 \\ 0 & 0 & 1 \\
\end{matrix}
\right) .
$$
Besides the arithmetic groups
$$
\Gamma=G^0({\ZZ}) \quad \hbox{\rm and} \quad \Gamma_1=G^0({\ZZ}) \cap \ker \det
$$
we shall consider the two congruence subgroups
$$
\Gamma[\sqrt{-3}]= \{ g \in \Gamma : g \equiv 1 \, (\bmod \, \sqrt{-3}\, )\}
$$
and
$$
\Gamma_1[\sqrt{-3}]
= \{ g \in \Gamma_1 : g \equiv 1 \, (\bmod \, \sqrt{-3} \, )\} \, .
$$
Any congruence subgroup $\Gamma'$ of $\Gamma$ acts properly discontinuously
on $B$ and the quotient is an orbifold, called a Picard modular surface. 
It is not compact, but can 
be compactified by adding finitely many cusps, that is, the orbits of 
$\Gamma'$ on the set $\partial B \cap {\PP}^2(F)$ 
of rational points in $\partial B$. 
It is well-known that the action of $\Gamma$ on $\partial B \cap {\PP}^2(F)$ 
is transitive (since the class number of ${\QQ}(\sqrt{-3})$ is $1$, 
see \cite{Z, Feustel}), 
so in this case there is one cusp. The group 
$\Gamma_1[\sqrt{-3}]$ instead, has four cusps; these are represented by
$(1:0:0)$, $(0:1:0)$, $(\rho:1:1)$ and $(\rho:1:-1)$ in $B \subset {\PP}^2$. 
We have an isomorphism
$$
\Gamma/\Gamma_1[\sqrt{-3}] \cong S_4 \times \mu_6
\quad g \mapsto (\sigma(g),\det (g)),
$$
with $\sigma(g)$ the permutation of the four cusps. Here $S_4$ denotes
the symmetric group on four objects and $\mu_6$ denotes the group
of $6$th roots of unity.

The stabilizer
in $G$ of the cusp $(1:0:0)$ is the parabolic group $P$ consisting
of matrices of the form
$$
\left( \begin{matrix}
t_1 & 0 & 0 \\ 0 & t_2 & 0 \\ 0 & 0 & t_3 \\ 
\end{matrix}\right)
\left( \begin{matrix} 
1 & x & -\bar{y} \\ 0 & 1 & 0 \\ 0 & y & 1 \\
\end{matrix}\right)
$$
with $t_i \in F^*$, $x,y \in F$ satisfying
$$
t_1\bar{t}_2 = t_3\bar{t}_3 , \quad x+\bar{x}= -y \bar{y}\, .
$$
Let $T$ be the corresponding maximal torus and $U$ the unipotent 
radical of $P$. Then $U$ has a filtration
$$
0 \to W \longrightarrow U \longrightarrow \tilde{U} \to 0,
$$
with $\tilde{U}={\rm R}_{F/{\QQ}}({\GG}_a)$ and the projection ${\rm pr}:
U \to \tilde{U}$ defined by 
$$
\left( \begin{matrix} 
1 & x & -\bar{y} \\ 0 & 1 & 0 \\ 0 & y & 1 \\\end{matrix}\right)
\mapsto y \, .
$$
The subgroup $W$ is given by $\{ x\in F: x+\bar{x}=0\}$ and is 
the center of $U$. The action of $U$ on $B$ is by translations
$$
(u,v) \mapsto (u+y,v+x-\bar{y}u).
$$
For the group $\Gamma$ the unipotent radical has $W=\sqrt{-3}\, {\ZZ}$
and $\tilde{U}=O_F$ (in fact, for given $y\in O_F$ take 
$x= \rho \, {\rm N}(y)+m \sqrt{-3}$), 
while the congruence subgroup $\Gamma_1[\sqrt{-3}]$
has the same $W$ and $\tilde{U}=\sqrt{-3}\, O_F$. 

The cusps of $\Gamma_1[\sqrt{-3}]\backslash B$ 
are singular points, but can be resolved by elliptic curves
${\CC}/\sqrt{-3}O_F$, cf.\ \cite{Holzapfel1, Feustel}.
\end{section}
\begin{section}{Fourier-Jacobi expansion of scalar-valued modular forms}
\label{FJ}
A scalar-valued 
modular form on the congruence subgroup $\Gamma_1[\sqrt{-3}]$ of $\Gamma$
is invariant under the translations in the unipotent radical  $\Gamma_1[\sqrt{-3}] \cap U$
of the parabolic subgroup that fixes the cusp $(1:0:0)$;
in particular it is invariant under the translations 
$v \mapsto v + m \sqrt{-3}$ with $m\in {\ZZ}$ of $W$ (cf.\ preceding section),  
thus giving rise to a Fourier-Jacobi expansion
$$
f(u,v)= \sum_{n=0}^{\infty} f_n(u) \, w^n \quad \hbox{\rm with} 
\quad w=e^{2\pi  v/\sqrt{3}}.
$$
Here the function $f_n$ defines 
a section of a line bundle $L^{\otimes n}$ on the elliptic
curve $E={\CC}/\sqrt{-3}O_F$. More precisely,
the function $f_n$ satisfies for all $\xi \in \sqrt{-3}\,  O_F$ the
relation
$$
f_n(u+\xi)= \exp(2\pi  n (\bar{\xi} u - \rho \xi\bar{\xi})/\sqrt{3}) \, f_n(u)\, . 
$$
Let $L$ be the line bundle on the elliptic curve 
$E$ corresponding to this factor of automorphy for $n=1$. 
It is the line bundle defined by the divisor class of degree $3$ on $E$ 
represented by $O_F/\sqrt{-3}\, O_F$ in ${\CC}/\sqrt{-3}\, O_F$. 
We know that $\dim H^0(E,L^{\otimes n})=3n$ 
for $n\geq 1$.  A generator $-\rho^2\in \mu_6$ acts on the
space of sections of $L$ with eigenvalues $-\rho^2, 1, -1$. We choose a basis
of eigenfunctions $X$, $Y+Z$ and $Y-Z$ of 
$\Gamma(E,L)$ for this action of $\mu_6$
with $X,Y,Z$ as in Finis, \cite{Finis1} p.\ 157.
Then the sections $X,Y,Z$ satisfy the equation $X^3=\rho(Y^3-Z^3)$.
We have a standard basis of $H^0(E,L^n)$ 
$$
\{ X^a Y^b Z^c: 0 \leq a \leq 2, \, 0 \leq b \leq n-a, \, a+b+c=n\} .
$$ 
The endomorphism ring $O_F$ of $E$ acts on the space $H^0(E,L^n)$
via the so-called Shintani operators
$$
m_{\alpha} : H^0(E,L^n) \to H^0(E, L^{n{\rm N}(\alpha)}), \quad
s(z) \mapsto s(\alpha \, z) \, .
$$
There are also operators in the other direction
$$
t_{\alpha}: H^0(E, L^{n{\rm N}(\alpha)}) \to  H^0(E,L^n),
$$
given by
$$
s(z) \mapsto \sum_{c} s(\alpha^{-1}(z+c)) e^{2\pi  n (\rho {\rm N}(c) -\bar{c} z)/\sqrt{3})},
$$
where $c$ runs over a complete set of representatives for 
$\sqrt{-3}\, O_F/\alpha \, \sqrt{-3} \, O_F$. We refer to
\cite{Finis1} and the literature given there.

As the isotropy group of a cusp in $S_4$ is isomorphic to $S_3=\{ \sigma \in S_4: \sigma(1)=1\}$ we find
an action of $S_3$ on the Fourier-Jacobi expansion of a Picard modular form.
This action is given by $(X,Y,Z) \mapsto (-X,Z,Y)$ for $R_2\sim (34)$ 
(Finis' notation in \cite{Finis1}, p.\ 153) 
and $(X,Y,Z) \mapsto (X,\rho Y,\rho^2 Z)$ for $R_3\sim (234)$.

\end{section}
\begin{section}{The Fourier-Jacobi expansion for vector-valued modular forms}
\label{FJ-vector}
For a vector-valued modular form on $\Gamma_1[\sqrt{-3}]$
the invariance under the unipotent radical
given in Section \ref{ourcase}
implies that in the Fourier-Jacobi expansion the functions $f_n$ satisfy the
relation
$$
f(u+\xi,v+\eta-\bar{\xi}u)={\rm Sym}^j
\left(\begin{matrix} 1 & \bar{\xi} \\ 0 & 1 \\ \end{matrix} \right)
f(u,v)
$$
and this implies that
$$
f_n(u+\xi)=\exp\left(2\pi  n (\bar{\xi}u-\rho\, \bar{\xi}\xi)/\sqrt{3}\right) \,
{\rm Sym}^j
\left(\begin{matrix} 1 & \bar{\xi} \\ 0 & 1 \\ \end{matrix} \right)
f_n(u) \eqno(4)
$$
The functions $f_n=(f_n^{(1)},\ldots,f_n^{(j+1)})$ 
represent sections of a vector bundle $A_n$ of rank $j+1$
on the elliptic curve. The vector bundle  $A_n$ is a tensor product
$L^{\otimes n} \otimes {\rm Sym}^j A$ with $A$ given on ${\CC}$ by the cocycle
$$
\xi \mapsto \left(\begin{matrix} 1 & \bar{\xi} \\ 0 & 1 \\ \end{matrix} \right)
$$
This vector bundle $A$ is an indecomposable bundle and hence $A_n$
has a filtration with $j+1$ quotients isomorphic to $L^{\otimes n}$. 

\begin{proposition}\label{existenceEisenstein}
Let $f$ be a vector-valued modular form of weight $(j,k)$ and 
character ${\rm det}^{\ell}$ on $\Gamma[\sqrt{-3}]$. 
If $f$ is not zero then $j\equiv k \, (\bmod \, 3)$. Moreover,
$f$ is a cusp form if $\ell \not\equiv j \, (\bmod \, 3)$.
\end{proposition}
\begin{proof} The first statement follows by looking at the action
of $\rho \, 1_3$. For the second we 
write $f=\sum_n f_n w^n$. The equation (4) implies that in the 
constant vector $f_0$ all but the first coordinate are zero. 
If we apply ${\rm diag}(1,1,\rho)$ we find that
$f_0(u)= f_0(\rho u)=\rho^{\ell}\,  {\rm Sym}^j ({\rm diag}(\rho^2,1)) f_0(u)$ 
implying that if $f_0\neq 0$ we must have $\ell \equiv j \, (\bmod \, 3)$.
\end{proof}
\end{section}
\begin{section}{The Hecke Rings}
Finis analyzed in \cite{Finis1} the Hecke rings for the arithmetic
groups $\Gamma$ and $\Gamma_1[\sqrt{-3}]$. These Hecke rings are the same
outside $3$ and are generated by operators $T(\nu)$, $T(\nu,\nu)$
for $\nu \in O_F$ with ${\rm N}(\nu)=p$, a prime congruent to $1 \, \bmod\, 3$
and operators $T(p)$, $T(p,p)$ for primes $p\equiv 2 \, (\bmod\, 3)$ in $O_F$
and for $\Gamma$ operators $T(\sqrt{-3})$ and $T(\sqrt{-3},\sqrt{-3})$.

If $\Gamma g \Gamma$ is a double coset
that can be written as a finite disjoint union of left cosets $\sum \Gamma g_i$,
then the action of the corresponding operator $T$ 
on modular forms in $M_{j,k}(\Gamma)$ is given by
$$
T\, f = \det(g) \, \eta(g)^{k-3} \sum_i f_{|j,k} \, g_i\, .
$$
There is a Petersson scalar product for pairs $(f,g)$ in $M_{j,k}(\Gamma)$
such that one of them is a cusp form. The Hecke operator $T(\nu)$
(resp.\ $T(\nu,\nu)$ with ${\rm N}(\nu)=p\equiv 1(\bmod \, 3)$ 
is adjoint with $T(\bar{\nu})$ (resp.\
$T(\bar{\nu},\bar{\nu})$) and for $p\equiv 2\, (\bmod \, 3)$
$T(p)$ (resp.\ $T(p,p)$) is self-adjoint. As a result, these Hecke operators
are simultaneously diagonalizable and the eigenvalues $\lambda_{\nu}$ 
of an eigenform are algebraic integers if $k\geq 3$ satisfying
$\lambda_{\bar{\nu}}=\bar{\lambda}_{\nu}$.

For the congruence subgroup $\Gamma_1[\sqrt{-3}]$
we need Hecke operators $T(\nu)$ and $T(p)$
for primes congruent to $1 \, (\bmod \, 3)$.
We shall write $T_{\nu}$ for $T(\nu)$ and $T_{-p}$ for $T(-p)$  
for rational primes $p\equiv 2 \, (\bmod \, 3)$.
For our calculations we now need disjoint left coset decompositions 
of the double cosets representing the Hecke operators 
$T_{\nu}$ for $\nu$ with ${\rm N}(\nu)=p$, a prime
$\equiv 1 (\bmod\, 3)$ and the $T_{-p}$ for the primes $p\equiv 2 \, 
(\bmod \, 3)$.

\begin{lemma} (Finis \cite{Finis1})
Write $\Gamma'=\Gamma_1[\sqrt{-3}]$. Then for $\nu\equiv 1 (\bmod \, 3)$ 
the operator $T_{\nu}$ is represented by
$$
\begin{aligned}
T_{\nu} &= \Gamma' \left(\begin{smallmatrix}
1 & 0 & 0 \\ 0 & p & 0 \\ 0 & 0 & \nu 
\end{smallmatrix}\right) \Gamma' \\ 
&= 
\Gamma' \left(\begin{smallmatrix}
p & 0 & 0 \\ 0 & 1 & 0 \\ 0 & 0 & \nu \end{smallmatrix}\right)
\bigoplus\oplus_{a,c}\Gamma'\left(\begin{smallmatrix}1 &
a & c \\ 0 & p & 0 \\ 0 & -\nu \bar{c} &
\nu \end{smallmatrix}\right)
\bigoplus \oplus_{b}\Gamma'\left(\begin{smallmatrix}
\nu & A(b) & -\bar{b} \\ 0 & \nu & 0 \\ 0 & b &
\bar{\nu}\end{smallmatrix}\right)\, ,
\end{aligned}
$$
where $(a,c)$ runs through the set of pairs $\{ (\rho {\rm N}(c)+\sqrt{-3}n,c):
c \in O_F \bmod\,  (\nu), n \in {\ZZ} \bmod \, p \}$ and $b$ through $O_F \bmod (\nu)$
and the algebraic integer $A(b)\in O_F$ is uniquely determined
$\bmod \, \nu$ by ${\rm Tr}(A(b)\bar{\nu})=-{\rm N}(b)$.

Moreover, for a prime $p\equiv 2 \, (\bmod \, 3) $ the operator $T_{-p}$ is 
represented by
$$
\begin{aligned}
T_{-p}&=
\Gamma' \left(\begin{smallmatrix}1 & 0 & 0 \\ 0 & p^2 & 0 \\
0 & 0 & -p\end{smallmatrix}\right)\Gamma'\\
&=\Gamma' \left(\begin{smallmatrix}p^2 & 0 & 0 \\ 0 & 1 & 0
\\ 0 & 0 & -p\end{smallmatrix}\right)
\bigoplus\oplus_{a,c}\Gamma'\left(\begin{smallmatrix}1 &
a & c \\ 0 & p^2 & 0 \\ 0 & p\bar{c} & -p \end{smallmatrix}\right)
\bigoplus\oplus_{m}\Gamma'\left(\begin{smallmatrix} -p &
\sqrt{-3}m & 0 \\ 0 & -p & 0 \\ 0 & 0 & -p\end{smallmatrix}\right) \, ,
\end{aligned}
$$
where $(a,c)$ runs through the set  $\{ (\rho {\rm N}(c)+\sqrt{-3}n,c): 
c\in O_F \bmod \, (p),\,  c\equiv 0 \, (\bmod \sqrt{-3}), n\in {\ZZ} \bmod \, p^2 \} $
and  $m$ through $({\ZZ}/p \ZZ)^*$.
\end{lemma}

Shintani and Finis showed how the Hecke operators act on the Fourier-Jacobi
expansion in the scalar-valued case. 
We quote from Finis. Let $f=\sum_n f_n w^n$ be the Fourier-Jacobi
expansion and let $g=T_{\nu}\, f$ (resp.\ $g=T_{-p}\, f$) have
Fourier-Jacobi expansion $g=\sum_n g_n w^n$, then we have in case $T=T_{\nu}$
for $g_n$ the expression
$$
g_n= \nu p^{k-2} m_{\nu}(f_{n/p})+ 
\nu^{-1} t_{\bar{\nu}}(f_{np})
+\bar{\nu}^{k-2} \nu^{-1} t_{\nu}m_{\bar{\nu}}(f_n)\, ,
$$
while in case $T=T_{-p}$ we have the expression
$$
g_n=(-p)^{2k-3}m_{-p}(f_{n/p^2})+
(-p)^{k-3}(p{\bf 1}_{\ZZ}(n/p)-1)f_n-t_{-p}(f_{np^2})/p\, .
$$
with ${\bf 1}_{\ZZ}$ the characteristic function of ${\ZZ}$ and $f_m=0$
if $m\not\in {\ZZ}$.
Note that $t_{\nu}$ is defined in Section~\ref{FJ}.

We now give a partial analogue for the vector-valued case. We write 
$f= \sum_n f_n w^n$ and $g=T\, f = \sum_n g_n w^n$, where
$$
f_n= \left(\begin{matrix} f_n^{(1)} \\ \vdots \\ f_n^{(j+1)}\\ \end{matrix} 
\right)
\qquad \hbox{\rm and} \qquad 
g_n= \left(\begin{matrix} g_n^{(1)} \\ \vdots \\ g_n^{(j+1)}\\ \end{matrix}
\right)
$$
We give the action on the last coordinate.
\begin{lemma}
We have for $T=T_{\nu}$ with ${\rm N}(\nu)=p\equiv 1 (\bmod \, 3)$ a prime
$$
g_n^{(j+1)}=\nu p^{k-2}\left(
p^j m_{\nu} f_{n/p}^{(j+1)} + p^{1-k} t_{\bar{\nu}}f_{np}^{(j+1)}+
\nu^{j-k} t_{\nu}m_{\bar{\nu}}(f_n^{(j+1)} \right),
$$
where we put $f_{n/p}^{(j+1)}=0$ if $n/p \not\in {\ZZ}$.

For $T_{-p}$ with $p$ a prime $\equiv 2 \, (\bmod \, 3)$ we have for 
$g_n^{(j+1)}$ the expression
$$
(-p)^{k+j-3}\left(p\, {\bf 1}_{\ZZ}(n/p) -1\right) f_n^{(j+1)}
-p^{2j+2k-3}m_{-p}(f_{n/p^2}^{(j+1)})-t_{-p}(f_{np^2}^{(j+1)})/p\, .
$$
\end{lemma}
\begin{proof}
Since the left coset representatives $g$ of the $T_{\nu}$ and $T_{-p}$
acts by upper triangular factors of automorphy 
we can express $g_n^{(j+1)}$ solely in terms of
the last component of $f_n$. An explicit calculation gives the result.
\end{proof}
\end{section}
\begin{section}{The Ring of Scalar-valued Picard Modular forms}\label{ringofscalarmodforms}
We recall the structure of the rings of modular forms on $\Gamma_1[\sqrt{-3}]$,
$\Gamma[\sqrt{-3}]$ and $\Gamma$ as obtained by Feustel and Holzapfel, 
cf.\ also \cite{Finis1}. 
The ring $M(\Gamma[\sqrt{-3}])$ is polynomial ring
$$
M(\Gamma[\sqrt{-3}])={\CC}[\varphi_0,\varphi_1,\varphi_2]
$$
with $\varphi_i\in M_3(\Gamma[\sqrt{-3}])=M_{0,3}(\Gamma[\sqrt{-3}])$ 
given by their Fourier-Jacobi 
expansions; in fact, $\varphi_{\nu} = \vartheta_{\nu}^3$ for $\nu=0,1,2$ with
$$
\vartheta_{\nu}= \sum_{\xi \in O_F} 
\rho^{-\nu \, {\rm Tr}(\xi)} \, m_{\xi}(Y) \, w^{{\rm N}(\xi)}\, . 
$$
Here $m_{\xi}$ is the endomorphism of $\oplus_n H^0(E,L^{\otimes n})$ defined
in Section \ref{FJ}; 
we have
$$
\begin{aligned}
\varphi_0=
1+  \left(9\, Y+9 \, Z \right) w+  \left(27\, {Y}^{2}+54\,YZ+ 27\, {Z}^{2}
 \right) {w}^{2}+ &\\
\left( 36\,{Y}^{3}+81\,{Y}^{2}Z+81\,Y{Z}^{2}+36\,{Z}
^{3} \right) {w}^{3}+ \ldots & \\
\end{aligned}
$$
The expansions of $\varphi_i$ are obtained by substituting 
$(\rho^i Y, \rho^{2i}Z)$ for $(Y,Z)$, as follows from the definition of 
$\vartheta_{\nu}$.

\smallskip
\noindent
{\sl Notation}
Before we proceed a word about our notation for representations of $S_4$.
The irreducible representations of $S_4$ correspond to the partitions of $4$
and are denoted by $s[4], s[3,1], s[2,2], s[2,1,1]$ and $s[1,1,1,1]$.
They are of dimensions $1,3,2,3,1$. 
Here $s[4]$ is the trivial and $s[1,1,1,1]$ the alternating representation.
The representations $s[3,1]$ is given by the permutation representation
on $\sum_{i=1}^4 x_i=0$ in $(x_1,\ldots,x_4)$-space.

\smallskip
  The group $\Gamma/\Gamma[\sqrt{-3}]
\cong S_4 \times \mu_2$ acts on $M_3(\Gamma[\sqrt{-3}])$; the generator
of $\mu_2$ acts by $-1$ on this space, while the representation of $S_4$
is the irreducible representation $s[2,1,1]$. More precisely, define 
forms $x_1,\ldots, x_4$ in $M_3(\Gamma[\sqrt{-3}])$ by
$$
\varphi_0+\varphi_1+\varphi_2, \quad
-3\varphi_0+\varphi_1+\varphi_2, \quad
\varphi_0-3\varphi_1+\varphi_2, \quad
\varphi_0+\varphi_1-3\varphi_2 . 
$$
In this way we have generators $x_1,\ldots, x_4$ with $\sum x_i=0$ 
and $\sigma \in S_4$
acts by $x_i \mapsto {\rm sign}(\sigma)\, x_{\sigma(i)}$.

   The ring $M(\Gamma_1[\sqrt{-3}])$ of modular forms on $\Gamma_1[\sqrt{-3}]$
is an extension of degree $3$ of $M(\Gamma[\sqrt{-3}])$ by a modular form
$$
\zeta \in S_6(\Gamma[\sqrt{-3}],{\rm det})
$$
satisfying a relation
$$
\zeta^3= \frac{-\rho}{\sqrt{-3} \, 3^7} 
 \, \varphi_0\varphi_1\varphi_2(\varphi_1-\varphi_0)(\varphi_2-\varphi_0)
(\varphi_2-\varphi_1) . \eqno(5)
$$
In fact, $\zeta$ is given by its Fourier-Jacobi expansion
$$
(1/6) \sum_{\xi \in O_F} {\xi}^5 \, m_{\xi}(X) \, w^{{\rm N}(\xi)}.
$$
Concretely,
$$
\begin{aligned}
\zeta = X \, w-27\,XYZ\, {w}^{3}+ \left( 32\,X{Y}^{3}+32\,X{Z}^{3} \right) {w}^{4}+ &\\
 \left( -211\,X{Y}^{6}+136\,X{Y}^{3}{Z}^{3}-211\,X{Z}^{6} \right) {w}^
{7}& +\ldots \\
\end{aligned}
$$
The action of $S_4$ on $\zeta$ is by the sign character.

Since $-1_3$ acts on $M_k(\Gamma_1[\sqrt{-3}])$ by $(-1)^k$ we find the 
decomposition under $\Gamma/\Gamma_1[\sqrt{-3}]$
$$
M_k(\Gamma_1[\sqrt{-3}])=M_k(\Gamma[\sqrt{-3}])\oplus 
M_k(\Gamma[\sqrt{-3}],{\rm det}) \oplus M_k(\Gamma[\sqrt{-3}],{\rm det}^2)
$$
with the recursions
$$
M_k(\Gamma[\sqrt{-3}],{\rm det}) =M_{k-6}(\Gamma[\sqrt{-3}]) \, \zeta , 
$$
and
$$
M_k(\Gamma[\sqrt{-3}],{\rm det}^2) =
M_{k-12}(\Gamma[\sqrt{-3}]) \, \zeta^2\, .
$$
Moreover, we have for $\ell=1,2$ (cf.\ Proposition \ref{existenceEisenstein})
$$
M_k(\Gamma[\sqrt{-3}],{\rm det}^{\ell})=S_k(\Gamma[\sqrt{-3}],{\rm det}^{\ell})\, .
$$
The ring $M(\Gamma)$ equals the ring of invariants
${\CC}[\varphi_0,\varphi_1,\varphi_2]^{S_4\times \mu_2}$
and is a polynomial ring generated by elements $\sigma_2$, $\sigma_4$ 
and $\sigma_3^2$ of weight $6$, $12$ and $18$.

The ring $M(\Gamma_1)$ is the ring of invariants 
${\CC}[\varphi_0,\varphi_1,\varphi_2,\zeta]^{S_4}$
and is the quotient of the ring
${\CC}[\sigma_2,\sigma_4,\sigma_3^2,\zeta \sigma_3,\zeta^2]$ 
by the ideal of relations implied by (5) and the notation (i.e.\
$(\zeta\sigma_3)^2= \zeta^2 \sigma_3^2$).
 
The Satake compactification
$\Gamma[\sqrt{-3}]\backslash B^*$ of the ball quotient
$\Gamma[\sqrt{-3}]\backslash B$
is isomorphic to ${\PP}^2= {\rm Proj}\,  {\CC}[\varphi_0,\varphi_1,\varphi_2]$, 
see \cite{Holzapfel1}.
Viewing ${\PP}^2$ as the hyperplane $\sum_{i=1}^4 x_i=0$ in ${\PP}^3$
we have six lines (viz.\ $x_i=x_j$), that make up the divisor of $\zeta^3$.

The cusps are the four points (in the $\varphi_i$-coordinates)
$$
c_1=(1:1:1), \, c_2=(1:0:0), \, c_3=(0:1:0), \, c_4=(0:0:1)
$$ 
The surface $\Gamma_1[\sqrt{-3}]\backslash B^*$ 
is a degree $3$ cover of $\Gamma[\sqrt{-3}]\backslash B^*$
branched along the union of the six lines $x_i-x_j=0$ with $1\leq i<j \leq 4$.
This surface has three singular points (order three quotient singularities)
corresponding to the
three intersections of these lines outside the four cusps, viz.\
$p_{14,23}=(1:1:0)$, $p_{13,24}=(1:0:1)$ and $p_{12,34}=(0:1:1)$.
In $p_{ij,kl}$ we have $x_i=x_j$ and $x_k=x_l$. 

\begin{remark}
The eigenvalues of the action of $T_{\nu}$ on $\varphi_i$ 
for $\nu$ with $\nu\bar{\nu}=p\equiv 1 \, (\bmod \, 3)$ are 
$(p+1)\nu + \bar{\nu}^2$ and for $T_{-p}$ these are $-1-p^3$,
cf.\ the formulas (9a) and (9b) below.
\end{remark}

\end{section}
\begin{section}{Expansion of Picard modular Forms along a Modular Curve}
\label{sec:restriction}
Picard modular surfaces contain many modular curves. In the following 
we shall need only one curve, namely the one that in the moduli space 
interpretation corresponds to the
degree $3$ covers of ${\PP}^1$ that are hyperelliptic curves of genus $3$.
This curve consists of six irreducible components and is defined as follows.

Let ${\mathcal H} \to B$ be the embedding of the upper half-plane in $B$
given by $\tau \mapsto (0,\sqrt{-3} \, \tau)$. The corresponding embedding
of algebraic groups ${\rm GL}_2 \to G$ is given by
$$
\left(
\begin{matrix}
a & b \\ c & d \\
\end{matrix} \right)
\mapsto
\left(
\begin{matrix}
a & \sqrt{-3} b & 0 \\
c/\sqrt{-3} & d & 0 \\
0 & 0 & ad-bc \\
\end{matrix}
\right)
$$
The image defines an algebraic curve in the Satake compactification of
$\Gamma_1[\sqrt{-3}]\backslash B$ and $\Gamma[\sqrt{-3}]\backslash B$
and it passes through two cusps. Using the action of $S_4$ we get six
curves $C_{ij}$ with $1 \leq i < j \leq 4$ on 
$\Gamma_1[\sqrt{-3}]\backslash B^*$ and six image curves on 
$\Gamma [\sqrt{-3}]\backslash B^*$. On the latter surface these
curves are given by $x_i=x_j$ as the next lemma shows.

\begin{lemma}
The stabilizer in $\Gamma=G^0(O_F)$ of the modular curve $C=C_{34}$ given by
$u=0$ equals
$$
\left\{ 
g = \left( 
\begin{matrix}
a & \sqrt{-3} \, b & 0 \\ c/\sqrt{-3} & d & 0 \\
0 & 0 & \varepsilon \\
\end{matrix} \right) : \left( \begin{matrix} a & b \\ c & d \\ \end{matrix} \right)
\in \Gamma_0(3), \varepsilon \in O_F^* \right\}
$$
\end{lemma}

If $f$ is a scalar-valued modular form of weight $(0,k)$ and character ${\rm det}^\ell$ then we can develop $f$ in a Taylor expansion along the curve $u=0$
$$
f(u,\sqrt{-3}\tau )=\sum_{n=0}^{\infty} f_n(\tau) \, u^n \, . \eqno(6)
$$
The functional equation of $f$
implies the following proposition.

\begin{proposition}
The coefficients $f_n(\tau)$ of $f$ in (6) are modular forms of weight $k+n$
on $\Gamma_1(3)$ and cusp forms for $n>0$. Moreover, $f_n=0$ unless
$n \equiv \ell \, (\bmod \, 3)$.
\end{proposition}

\begin{example} Writing $\varphi_i=\sum \varphi_{i,n} u^n$ we have
$$
\begin{aligned}
\varphi_0 & = \varphi_{0,0}+\varphi_{0,6} u^6+O(u^{12}), \\
\varphi_1 & = \varphi_{1,0}+\varphi_{1,3} u^3+O(u^{6}), \\
\varphi_2 & = \varphi_{1,0}-\varphi_{1,3} u^3+O(u^{6}), \\
\end{aligned}
$$
with 
$$
\begin{aligned}
\varphi_{0,0}=& 1 + 18\, q + 108\, q^2 + 234\, q^3 + 234\, q^4 + O(q^5), \\
\varphi_{1,0}=& \varphi_{2,0}=\sqrt{-3} \, (
1 - 9\, q + 27\, q^2 - 9\, q^3  - 117\, q^4 + O(q^5)).\\
\end{aligned}
$$
For the modular form $\zeta$ of weight $6$ we have 
$$
\zeta(u,\sqrt{-3}\tau)=\zeta_1 \, u+ \zeta_7 \, u^7+ O(u^{13}),
$$
with $\zeta_1 \in S_7(\Gamma_1(3))$ and $\zeta_7 \in S_{13}(\Gamma_1(3))$.
\end{example}

Similarly, we can develop vector-valued modular forms along the curve $C$.
We write such a modular form $F \in M_{j,k}(\Gamma[\sqrt{-3}],{\rm det}^{\ell})$ 
as
$$
F(u,\sqrt{-3}\tau)= \sum_{n=0}^{\infty} \left( \begin{matrix}
F_n^{(1)} \\
\vdots \\
F_n^{(j+1)}\\
\end{matrix} \right) \, u^n \, .
$$
\begin{proposition}\label{expansion2}
The first component $F_n^{(1)}$ is a modular form of weight $k+n$ on 
$\Gamma_1(3)$ and a cusp form if $n>0$. Moreover $F_n^{(m)}$ 
vanishes unless $ n +(j+1-m) \equiv \ell \, (\bmod \, 3)$. The function
$F_0^{(m)}$ is a modular form of weight $k+m-1$ on $\Gamma_1(3)$, while
for $n>0$ the function 
$F_n^{(m)}$ is a quasi-modular form of weight $k+m+n-1$ on $\Gamma_1(3)$.
\end{proposition}
\begin{proof}
We refer to \cite{K-Z} for the definition of a quasi-modular form.
The proof of the first statement follows from writing out the 
transformation behavior. 
The second statement follows by applying ${\rm diag}(1,1,\rho)$.
\end{proof}
\end{section}

\begin{section}{Rankin-Cohen Brackets}\label{sec:Rankin-Cohen}
We now construct vector-valued modular forms by a variant of the
Rankin-Cohen brackets. Recall that we have the relation
$$
j_2(g,u,v)^{-1}= j_1(g,u,v) (J(g,u,v))^t,
$$
between the automorphy factors and the Jacobian of the group action on the 
ball $B$. This implies that for a differentiable function $f: B \to {\CC}$
with gradient
$$
\nabla f = \left( \begin{smallmatrix} \frac{\partial f}{\partial u} \\
\\ 
\frac{\partial f}{\partial v}
\end{smallmatrix}\right)
$$
we get using coordinates $b=(u,v)$ on $B$ and writing $g=(g_{ij}) \in G$
$$
\nabla \left( \frac{f(g\cdot b)}{j_1(g,b)^k} \right)= -k 
\frac{f(g\cdot b)}{j_1(g,b)^{k+1}} 
\left(\begin{matrix} g_{23}\\ g_{21}\end{matrix}\right) +
\frac{1}{j_1(g,b)^{k+1}} 
 j_2(g,b)^{-1}  \nabla f(g\cdot b) \, .
$$
We can get rid of the first term on the right hand side 
by using a bracket. 

\begin{definition}\label{crochet}
For $k, l\in {\ZZ}_{\geq 1}$ and $f, h: B\to \CC$
differentiable functions we put
$$
\left[f,h\right]_{k,l}(b)=
\frac{1}{l} f(b)\nabla h(b) - \frac{1}{k} h(b)\nabla f(b)\, .
$$
\end{definition}

A straightforward computation leads to the following proposition.

\begin{proposition}\label{crochetprop}
For every unitary similitude $g$ and functions $f,h: B \to {\CC}$ 
we have
\begin{align*}
\left[f\vert_k g,h\vert_l g\right]_{k,l}(b)=
&j_1(g,b)^{-k-l-1}j_2(g,b)^{-1}\left[f,h\right]_{k,l}(g.b)\\
=&(\left[f,h\right]_{k,l}\vert_{1,k+l+1}\ g)(b).
\end{align*}
\end{proposition}

Let $\Gamma'$ be a finite index subgroup of the Picard modular group and let 
$\chi_1, \chi_2$ be finite order characters. Then we define the bracket
for $f\in M_{k}(\Gamma',\chi_1)$ and $h \in M_{l}(\Gamma',\chi_2)$ by
$$
[f,h]:= [f,h]_{k,l}\, .
$$

\begin{corollary}\label{cons}
Let $f\in M_{k}(\Gamma',\chi_1)$ and $h \in M_{l}(\Gamma',\chi_2)$ 
be scalar-valued Picard modular forms. Then 
$\left[f,h\right]$ is a vector-valued modular cusp form in 
$S_{1,k+l+1}(\Gamma',\chi_1\cdot \chi_2)$.
\end{corollary}
\begin{proof}
In view of Prop.\ \ref{crochetprop} 
the only thing to check is that we obtain a cusp form, that is, 
the Fourier-Jacobi expansions at the different cusps of the group 
have no constant terms. This is immediate as differentiation
kills constant terms in the Fourier-Jacobi expansions.
\end{proof}

\end{section}

\begin{section}{Modules of Vector-valued Picard Modular Forms}\label{modules}
We denote the vector space of Picard modular modular forms of weight $(j,k)$
on the group 
$\Gamma_1[\sqrt{-3}]$ (resp.\ $\Gamma[\sqrt{-3}]$)
by $M_{j,k}(\Gamma_1[\sqrt{-3}])$ (resp.\ by $M_{j,k}(\Gamma[\sqrt{-3}])$).
Note that $-1_3$ acts by $(-1)^k {\rm Sym}^j(-1_2)$;  we thus
can decompose $M_{j,k}(\Gamma_1[\sqrt{-3}])$ as
$$
M_{j,k}(\Gamma_1[\sqrt{-3}])=M_{j,k}(\Gamma[\sqrt{-3}])
\oplus M_{j,k}(\Gamma[\sqrt{-3}],\det{})
\oplus (\Gamma[\sqrt{-3}],\det{}^2).
$$
The corresponding spaces of cusp forms are denoted by $S_{j,k}$.
We thus have modules 
$$
{\mathcal M}_j={\mathcal M}_j^0\oplus {\mathcal M}_j^1 \oplus {\mathcal M}_j^2
$$ 
with
${\mathcal M}_j^{\ell}= \oplus_k M_{j,k}(\Gamma[\sqrt{-3}],\det{}^{\ell})$
and for the cusp forms
$$
\Sigma_j= \oplus_k S_{j,3k+j}(\Gamma_1[\sqrt{-3}])
=\Sigma_j^0\oplus \Sigma_j^1\oplus \Sigma_j^2 \, .
$$
Note that by Proposition \ref{existenceEisenstein} we have 
${\mathcal M}_j^{\ell}=
\Sigma_j^{\ell}$ if ${\ell}\not\equiv j \, (\bmod \, 3)$.
These are modules over
$M={\mathcal M}_0^0=
\oplus_r M_{3r}(\Gamma[\sqrt{-3}])={\CC}[\varphi_0, \varphi_1,\varphi_2]$.

By using the Hirzebruch-Riemann-Roch theorem one can show that for $j+3k>4$ 
we have
$$
\dim M_{j,j+3k}(\Gamma_1[\sqrt{-3}])= \frac{3\, (k-1) (j+1) (j+k)}{2}  
+ \frac{j(j+1)(j+2)}{3}+c
$$
where $c$ is a constant depending only on congruences for $j$.
We have $c=4,2,2,4$ for $j=0,1,2,3$. 
We refer to \cite{B-vdG} for this. 
In fact, we have by the holomorphic Lefschetz formula for
$j\equiv 2 \, (\bmod \, 3)$ the more precise formula
$$
\dim M_{j,j+3k}(\Gamma[\sqrt{-3}],{\rm det}^{\ell}) = 
\frac{j+1}{2}\, k^2+ \frac{j^2-1}{2} k  +c^{\prime}
$$
while for $j\not\equiv 2 \, (\bmod \, 3)$ we have
$$
\dim M_{j,j+3k}(\Gamma[\sqrt{-3}],{\rm det}^{\ell})
=\frac{j+1}{2}\, k^2+ (\frac{j^2-1}{2}+\varepsilon) \, k + c^{\prime\prime}
$$
with $\varepsilon$ given by
$$
\varepsilon = \begin{cases}
2 & j\equiv \ell \, (\bmod \, 3), \ell\neq 2 \\
0 &  j \not\equiv \ell \, (\bmod \, 3), \ell \neq 2 \\
-2 & \ell=2 \\
\end{cases}
$$
and with $c'$ and $c^{\prime\prime}$ not depending on $k$.
In fact, given such formulas, our cohomological calculations in
\cite{B-vdG} determine the constants $c,c'$ and $c^{\prime \prime}$ for small $j$.
\end{section}

\begin{section}{Examples of Vector-valued Picard Modular Forms}\label{examples}

\begin{subsection}{Forms in $S_{1,7}$}

As a first example we consider the forms
$$
\Phi_0=-\frac{[\varphi_1,\varphi_2]}{6\pi \sqrt{-1}}, 
\quad  \Phi_1=-\frac{[\varphi_2,\varphi_0]}{6\pi \sqrt{-1}}, 
\quad \Phi_2=-\frac{[\varphi_0,\varphi_1]}{6\pi\sqrt{-1}} \, .
$$
By Corollary \ref{cons} 
these forms belong to $S_{1,7}(\Gamma[\sqrt{-3}])$ and they are linearly
independent as one sees by calculating the Fourier-Jacobi expansions,
see below.
Since the $\varphi_i$ generate the $S_4$-representation $s[2,1,1]$ the $\Phi_i$
generate the $S_4$-representation $\wedge^2 s[2,1,1]=s[2,1,1]$. 
To make the action of $S_4$ more transparent we define 
$$
X_1=\Phi_0+\Phi_1+\Phi_2, \quad X_2=-\Phi_0, \quad X_3=-\Phi_1,\quad X_4=-\Phi_2
$$
and observe that $\sum_{i=1}^4 X_i=0$ and 
the action of $\sigma \in S_4$ is by $X_i \mapsto
{\rm sgn}(\sigma) X_{\sigma(i)}$.

We find the Fourier-Jacobi expansions
$$
\left(\begin{matrix}
\Phi_0^{(1)} \\ \Phi_0^{(2)} \\ 
\end{matrix} \right) =
\left(\begin{matrix}
 \frac{\sqrt{3}}{2\pi} 
((Y'-Z')\, w -3(2YY'+ 3Y'Z- 3YZ'-2ZZ')\, w^2 + \ldots
\\
(Y-Z)\, w +(-6\, Y^2+6\, Z^2)\, w^2+ \ldots
\\
\end{matrix}\right)
$$
Here the primes refer to the derivative with respect to $u$.
The corresponding expansions for $\Phi_i$ are obtained from this one 
by substituting $(\rho^i Y,\rho^{2i}Z)$ for $(Y,Z)$.

We determine the expansion of the
$\Phi_i$ along the curve $C_{34}$ given by 
$\{(0,\sqrt{-3}\tau): \tau \in {\mathcal H}\} \subset B$.  We find
$$
\Phi_0(u,\sqrt{-3}v)=\left( \begin{matrix}
g_2 u^2+g_8 u^8 +O(u^{14}) \\
g_3 u^3+O(u^9)\\ \end{matrix} 
\right)
$$
with $g_2=q-15\, q^2+O(q^3) \in S_{9}(\Gamma_1(3))$, $g_8 \in S_{15}(\Gamma_1(3))$ and
$g_3$ a quasi-modular form of weight $11$ on $\Gamma_1(3)$. Moreover,
$$
\Phi_1(u,\sqrt{-3}v)=\left( \begin{matrix}
h_2 u^2 + O(u^5)\\
h_0+h_3u^3+O(u^6)\\
\end{matrix} \right) 
$$
and
$$
\Phi_2(u,\sqrt{-3}v)=\left( \begin{matrix}
h_2 u^2 + O(u^5)\\
-h_0+h_3u^3+O(u^6)\\
\end{matrix} \right) 
$$
with 
$h_0\in S_8^{\rm new}(\Gamma_0(3))$, $h_2 =q+12\, q^2 +O(q^3)
\in S_9(\Gamma_1(3))$ and $h_3$ quasi-modular of weight $11$ for $\Gamma_1(3))$.

By relation (1)  the form $\Phi_1\wedge \Phi_2$ is a scalar-valued 
modular form of weight $15$ and character ${\rm det}^2$.
Using the Fourier-Jacobi expansion we see that 
$$
\Phi_1\wedge \Phi_2= 2\pi \sqrt{-1}  (Y' Z - Z' Y)\, w^2  + \ldots
$$
and up to a factor $1/Z^2$ the coefficient of the
first term is the derivative of $Y/Z$ which
is not constant. 
Since $\Phi_1\wedge \Phi_2 \in S_{15}(\Gamma[\sqrt{-3}],{\rm det}^2)$,
it is divisible by $\zeta^2$; in fact, of the form $f\zeta^2$ with $f$ of weight $3$; using the action of $S_4$ we see that
there is a non-zero constant $c \in {\CC}$ such that
$$
\Phi_1\wedge \Phi_2=c \,  \zeta^2 \, \varphi_0 \, . \eqno(7)
$$
We now draw an important conclusion about the vanishing locus of the forms
$\Phi_i$ (or $X_i$).

\begin{corollary}\label{vanishingPhi}
The forms $X_i$ with $i=1,2,3,4$ do not vanish outside the union of the
modular curves $C_{ij}$. More precisely, the vanishing locus of $X_i$ 
consists of the three curves $C_{jk}$, $C_{jl}$ and $C_{kl}$ 
passing through cusp $c_i$.
\end{corollary}
\begin{proof}
From the expansions given above 
 we deduce that $\Phi_0$ vanishes on three
of the six $C_{ij}$. On the other three $C_{ij}$ the first component
vanishes, while second component is a non-zero modular forms of weight $8$
on $\Gamma_0(3)$. Since it vanishes on the intersections of the $C_{ij}$
we see that there cannot be more zeros in view of the formula for the number
of zeros of a modular form on $\Gamma_0(3)$.
\end{proof}

\end{subsection}
\begin{subsection}{Forms in $S_{1,7}(\Gamma[\sqrt{-3}],\det{})$}
\label{sec:det1-case}
Recall that according to \cite{Feustel} the form 
$\zeta$ satisfies the identity
$\zeta=c_{\zeta} \, \prod_{0\leq i \leq 5}
\vartheta_i\in S_{6}(\Gamma[\sqrt{-3}],{\rm{det}})$ 
with $c_{\zeta} \in {\CC}^*$
and $\varphi_k=\vartheta_k^3$ for $0\leq k \leq 2$.
We form the bracket with one of the forms 
$\varphi_k$ with $k=0,1,2$:
\begin{align*}
[\zeta,\varphi_k]&=c_{\zeta} 
[\prod_{0\leq i \leq 5}\vartheta_i,\vartheta_k^3]
=\frac{c_{\zeta}}{3}(\prod_{0\leq i \leq 5}\vartheta_i)\nabla \vartheta_k^3
-\frac{1}{6}\vartheta_k^3(\nabla (\prod_{0\leq i \leq 5}\vartheta_i))\\
&=c_{\zeta}\vartheta_k^3\Big((\prod_{\substack{0\leq i \leq 5\\ i\neq k}}\vartheta_i)\nabla\vartheta_k
-\frac{1}{6}\nabla (\prod_{0\leq i \leq 5}\vartheta_i)\Big).
\end{align*}
So we can divide by $\vartheta_k^3=\varphi_k$ to obtain
$[\zeta,\varphi_k]/\varphi_k \in S_{1,7}(\Gamma[\sqrt{-3}],{\rm{det}})$.
More generally we put
$$
\gamma_{ij}= \frac{1}{x_i-x_j}\, [\zeta,x_i-x_j] \qquad (\{ i,j,k,l\}=\{ 1,2,3,4\})
$$
and obtain thus six elements in $S_{1,7}(\Gamma[\sqrt{-3}],{\rm det})$
satisfying the relation
$$
\sum_{1\leq i < j \leq 4} \gamma_{ij}=0.
$$
These $\gamma_{ij}$ generate a $5$-dimensional space which decomposes as
$s[2,2]\oplus s[2,1,1]$ as $S_4$-representation. The $s[2,1,1]$-space is generated by
the four elements 
$a_i=\gamma_{jk}+\gamma_{jl}+\gamma_{kl}$ with $\sum a_i=0$, while the
$s[2,2]$-space is generated by the three 
$b_{ij,kl}=\gamma_{ij}+\gamma_{kl}$ satisfying $\sum b_{ij,kl}=0$.

The Fourier-Jacobi expansion of the second component of $6 \gamma_{12}$
is
$$
\begin{aligned}
&-Xw+18X(Y+Z) {w}^{2}- 27X( 2Y^{2}+YZ+2Z^{2} ) w^{3}+ \\
& 88X(Y^{3}+Z^{3}) {w}^{4} 
-18X ( 11Y^{4}-Y^3Z-YZ^3+11Z^{4} ) {w}^{5}+O \left( {w}^{6}  \right)&,\\
\end{aligned}
$$
and the expansion of $\gamma_{13}^{(2)}$ (resp.\ $\gamma_{14}^{(2)}$) is 
obtained by substituting $(\rho Y,\rho^2 Z)$ (resp.\ $(\rho^2 Y,\rho Z)$)
for $(Y,Z)$;
the expansion for $6 \gamma_{34}^{(2)}$
is
$$
Xw-6X(Y+Z)w^2-9X(2Y^2-3YZ+2Z^2)w^3 + O(w^4)
$$
and then $\gamma_{23}^{(2)}$ and $\gamma_{24}^{(2)}$ 
are obtained by subsituting
$(\rho^2 Y,\rho Z)$ (resp.\ $(\rho Y, \rho^2 Z)$).

\smallskip
The relation between the $\gamma_{ij}$ and the $\Phi_i$ is as follows.
\begin{lemma}\label{Phi-G}
We have 
$$
\frac{\zeta}{\varphi_1\varphi_2}\, \Phi_0 = 
-\frac{16\, \zeta X_2}{(x_3-x_1)(x_4-x_1)}=
\frac{1}{3\sqrt{-3}} \, (\gamma_{13}-\gamma_{14}) \, .
$$
\end{lemma}
\begin{proof}
The proof is just a computation.
\end{proof}
\end{subsection}

\begin{subsection}{Forms in $S_{1,7}(\Gamma[\sqrt{-3}],{\rm det}^2)$ 
and $S_{1,10}(\Gamma[\sqrt{-3}],{\rm det}^2)$}\label{sec:det2-case}

We start by defining a form $\Psi_1$ in $S_{1,7}(\Gamma[\sqrt{-3}],{\rm det}^2)$.
The form $\Psi_1$ is defined as the quotient of the projection of
$\varphi_0\varphi_1\Phi_2$ to the $s[1,1,1,1]$-space in 
$S_{1,13}(\Gamma[\sqrt{-3}])$ divided by $\zeta$:
$$
\Psi_1= \frac{
\varphi_0(\varphi_1-\varphi_0) \, \Phi_0 - 
\varphi_2 (\varphi_2-\varphi_1)\, \Phi_2}{\zeta}\, .
$$
This form behaves in the right way; the only thing to check is that it is
holomorphic and a cusp form. 
Since the divisor of zeta consists of the six curves $C_{ij}$
we have to check holomorphicity along these curves. But $\Psi_1$ is
$S_4$-invariant, hence it suffices to check this along one of the $C_{ij}$.
This can be read off from the Taylor expansion.
The forms $\Psi_1$ vanishes at the cusp~$\infty$. 

The Fourier-Jacobi expansion
of the second component of $\Psi_1$ is up to a non-zero factor
$$
X^2 \, w^2 -24\, X^2 YZ\, w^4 + 34\, X^2(Y^3+Z^3) \, w^5 -81 \, X^2Y^2Z^2 \, w^6+\ldots 
$$

The form $\Psi_2$ in $S_{1,10}(\Gamma[\sqrt{-3}],{\rm det}^2)$
is defined as $F/\zeta$ with $F$ given by
$$
\varphi_0 (\varphi_0-\varphi_1)(\varphi_0+\varphi_1-3\varphi_2)\, \Phi_0
-\varphi_2(\varphi_1-\varphi_2)(\varphi_1+\varphi_2-3\varphi_0)\, \Phi_2\, .
$$
and the second component of 
$\Psi_2$ has Fourier-Jacobi expansion (up to a non-zero factor)
$$
X^2 \, w^2 -6\, X^2YZ \, w^4 + 70 \, X^2(Y^3+Z^3)\, w^5 - 405\, X^2Y^2Z^2\, w^6+\ldots
$$
The group $S_4$ acts on $\Psi_2$ by the sign character.

We finish by calculating some wedge products.
The form $\Psi_1\wedge \Psi_2$ is an $S_4$-anti-invariant scalar-valued
modular form of weight $18$ with trivial character; in fact,
$$
\Psi_1\wedge \Psi_2= 2^2 3^7(\rho-1) c  \, \zeta^3  \eqno(8)
$$
with $c\in {\CC}^*$ given in (7).
One can also calculate the wedge of $\Psi_1$ with 
the space $S_{1,7}(\Gamma[\sqrt{-3}],\det{})$: we have
$$
\Psi_1 \wedge \gamma_{1j}= -\frac{c}{6\sqrt{-3}} 
\, \zeta^2(\varphi_0+\varphi_1+\varphi_2-2\varphi_{j+1})
\quad \hbox{\rm for $j=2,3,4$}.
$$
This shows, for example, 
that wedging by $\Psi_1$ annihilates the $s[2,2]$-subspace of
$S_{1,7}(\Gamma[\sqrt{-3}],{\rm det})$.
\end{subsection}
\end{section}
\begin{section}{Low Weight Eisenstein Series}\label{eisenstein}
In this section we construct Eisenstein series of low weight for our
Picard modular group $\Gamma[\sqrt{-3}]$. Note that by Proposition
\ref{existenceEisenstein} the weight of a non-trivial Eisenstein series
in $M_{j,k}(\Gamma[\sqrt{-3}],{\rm det}^{\ell})$ 
satisfies $j\equiv \ell \, (\bmod \, 3)$.
Eisenstein series exist if $j+k>4$, cf.\ e.g.\ \cite{Shimura2}. 

\begin{proposition} 
For $j+k>4$ and $j\equiv \ell \, (\bmod \, 3)$ the space of Eisenstein series in
$M_{j,k}(\Gamma[\sqrt{-3}]),\det{}^{\ell})$ has dimension $4$ 
and as a $S_4$-representation it is of the form 
$(s[4]\oplus s[3,1])\otimes s[1,1,1,1]^{\otimes k}$.
\end{proposition}
\begin{proof}
Since the group $S_4$ permutes the cusps it follows that the representation
is either $s[4]\oplus s[3,1]$ or $s[2,1,1] \oplus s[1,1,1,1]$. 
Let now $E\in M_{j,k}(\Gamma[\sqrt{-3}]),\det{}^{\ell})$ be an invariant 
or anti-invariant element in the space of
Eisenstein series under the action of $S_4$.
Then the matrix ${\rm diag}(-1,-1,1)=R_2$ 
corresponds to the transposition $(34)$ and acts
on an Eisenstein series $E$ by 
$$
{\rm diag} ((-1)^k,(-1)^{k+1},\ldots, (-1)^{k+j}).
$$ 
From the transformation rule (4) it follows that the constant term is a vector
$(c^{(1)},\ldots, c^{(j+1})^t$ with zero entries $c^{(m)}$ on places $m>1$,
cf.\ the proof of Proposition \ref{existenceEisenstein}.
Therefore the action on $E$ is by $(-1)^k$ and this proves the proposition. 
\end{proof}
In general the eigenvalue for $T_{\nu}$ with ${\rm N}(\nu)=p\equiv 1(\bmod \, 3)$
of an Eisenstein series of weight $(j,k)$ is
$$
(p^{k-2}+1) \nu^{j+1}+\bar{\nu}^{j+k-1}. \eqno(9a)
$$
and for $T_{-p}$ with $p$ a prime $\equiv 2 \, (\bmod \, 3)$
$$
(-1)^{j+1}\left( p^{2 k+j-3}+p^{j+1}+(-1)^k(p-1)p^{k+j-3}\right) \, . \eqno(9b)
$$

Now we look at the remaining cases with $j+k\leq 4$.

\begin{proposition} The non-zero Eisenstein spaces for $j+k\leq 4$ are given in the following table as representations of $S_4$.

\bigskip
\vbox{
\centerline{\def\quad{\hskip 0.3em\relax}
\vbox{\offinterlineskip
\hrule
\halign{&\vrule#& \quad \hfil#\hfil \strut \quad  \cr
height2pt&\omit&&\omit&&\omit&&\omit&&\omit& \cr
& $(j,k,\ell)$ && $(0,0,0)$ && $(0,3,0)$ && $(1,1,1)$ && $(2,2,2)$ & \cr
height2pt&\omit&&\omit&&\omit&&\omit&&\omit& \cr
\noalign{\hrule} 
height2pt&\omit&&\omit&&\omit&&\omit&&\omit& \cr
& ${\rm rep}$ && $s[4]$ && $s[2,1,1]$ && $s[1,1,1,1]$ && $s[4]$ &\cr 
height2pt&\omit&&\omit&&\omit&&\omit&&\omit& \cr
} \hrule}
}}

\end{proposition}
\begin{proof}
The cases with $(j,k)=(0,0)$ and $(0,3)$ are well-known, 
cf.\ Section~\ref{ringofscalarmodforms}.
The space $M_{1,1}(\Gamma[\sqrt{-3}],{\rm det})$ is generated by the form
$$
E_{1,1}= \Psi_1/\zeta
$$
with $\Psi_1$ the $S_4$-invariant form that 
generates $S_{1,7}(\Gamma[\sqrt{-3}],\det{}^2)$. 
In fact, after multiplication by $\zeta$ our Eisenstein series yields a cusp form 
of weight $(1,7)$ and character ${\rm det}^2$.
Since in view of (8) the form 
$\Psi_1$ does not vanish outside the curves $C_{ij}$ 
it suffices to check the divisiblity along the curve $C=C_{34}$. 
This form does not vanish at every cusp.  
Similarly, the space $M_{1,4}(\Gamma[\sqrt{-3}],\det{})$ is generated
by $\varphi_i E_{1,1}$ and by the invariant form $\Psi_2/\zeta$. Finally,
the space $M_{2,2}(\Gamma[\sqrt{-3}],\det{}^2)$ is generated by a form $K_2=
{\rm Sym}^2(E_{1,1})$.
\end{proof} 

\begin{remark} The eigenvalues of $E_{1,1}$ for the Hecke operators are given by
formula (9a) and (9b). Note that these eigenvalues are not integral.
\end{remark}
\end{section}
\begin{section}{The Structure of the Module ${\mathcal M}_1$ of Vector-valued Modular Forms}\label{structure}

We shall determine the structure of the $M$-module ${\mathcal M}_1$
of Picard modular forms.
We shall construct generators for the ${\mathcal M}_1^{\ell}$ and 
$\Sigma_1^{\ell}$. In order to see that these
generators exhaust $\Sigma_1$ we need the following dimension formula:
$$
\dim S_{1,3k+1}(\Gamma_1[\sqrt{-3}])= 3k^2-3 \qquad \hbox {\rm for $k\geq 1$}
\eqno(10)
$$
given in Section  \ref{modules}.
In fact, we have the more precise formula for $k\geq 1$ 
$$
\dim S_{1,3k+1}(\Gamma[\sqrt{-3}],{\rm det}^{\ell})=
\begin{cases}
k^2-1 & \ell =0 \\
(k+1)^2-4 & \ell =1 \\
(k-1)^2 & \ell =2 \\
\end{cases}  \eqno(11)
$$
but assuming (10) it will follow from our proof.
In fact, we will show the existence of a submodule of $\Sigma_1$
whose graded part of degree $k$ has dimension $3k^2-3$.
Then (10) shows that this exhausts all of $\Sigma_1$, hence we have
exhausted $\Sigma_1^i$ for $i=0,1,2$ as well. Thus (10)
will imply (11).

\begin{subsection}{The module ${\mathcal M}_{1}=\Sigma_1^0$}
We give a presentation of this module.

\begin{theorem}
There exist cusp forms $\Phi_i$ for $i=0,1,2$ of weight $(1,7)$ 
spanning an irreducible $S_4$-representation of type $s[2,1,1]$
that generate $\Sigma_1^0$
as $M$-module and the module of relations is generated over $M$
by the $S_4$-invariant relation
$$
\sum_{i=0}^2 \varphi_i \Phi_i =0 \, .
$$
\end{theorem}

In particular, we see that as a $S_4$-representation we have for $k\geq 1$ 
$$
S_{1,3k+7}(\Gamma[\sqrt{-3}])=s[2,1,1] \otimes {\rm Sym}^k(s[2,1,1])-
{\rm Sym}^{k-1}(s[2,1,1]).
$$
Here is a table\footnote{The meaning of the colors is: red indicates a form whose eigenvalues are given in the tables at the end, blue are lifts, green are Eisenstein series.} with the multiplicities of the irreducible representations
of $S_4$ in $S_{1,1+3k}(\Gamma[\sqrt{-3}])$.

\bigskip
\vbox{
\centerline{\def\quad{\hskip 0.3em\relax}
\vbox{\offinterlineskip
\hrule
\halign{&\vrule#& \quad \hfil#\hfil \strut \quad  \cr
height2pt&\omit&&\omit&&\omit&&\omit&&\omit&&\omit& \cr
& $k$ && $s[4]$ && $s[3,1]$ && $s[2,2]$ && $s[2,1,1]$ && s[1,1,1,1]& \cr
height2pt&\omit&&\omit&&\omit&&\omit&&\omit&&\omit& \cr
\noalign{\hrule}
height2pt&\omit&&\omit&&\omit&&\omit&&\omit&&\omit& \cr
& $0$ && $0$ && $0$ && $0$ && $0$ && $0$ & \cr
& $1$ && $0$ && $0$ && $0$ && $0$ && $0$ & \cr
& $2$ && $0$ && $0$ && $0$ && {\color{red}$1$} && $0$ & \cr
& $3$ && $0$ && {\color{red}$1$} && {\color{red}$1$} && {\color{red}$1$} && $0$ & \cr
& $4$ && $0$ && $2$ && $1$ && $2$ && $1$ & \cr
& $5$ && $1$ && $3$ && $2$ && $3$ && $1$ & \cr
height2pt&\omit&&\omit&&\omit&&\omit&&\omit&&\omit& \cr
} \hrule}
}}

\begin{proof}
The forms $\Phi_i$ defined in Section \ref{examples}
belong to $S_{1,7}(\Gamma[\sqrt{-3}])$ and are linearly 
independent.
As $M_3(\Gamma[\sqrt{-3}])$ is the $S_4$-representation $s[2,1,1]$
the $\Phi_i$  generate the irreducible representation $s[2,1,1]=
\wedge^2 s[2,1,1]$. 
We consider the kernel of the map of $M$-modules 
$M \otimes \langle \Phi_0, \Phi_1,\Phi_2\rangle \to \Sigma_1^0$
given by $f \otimes \Phi_i \mapsto f \, \Phi_i$. Take an irreducible
$S_4$-representation in the kernel. Suppose that it has dimension $\geq 2$.
Then we have two independent relations
$\sum_{i=0}^2 f_i \Phi_i=0$ and $\sum_{i=0}^2 g_i \Phi_i=0$. By suitably
subtracting and multiplying with elements of $M$ and using the action of $S_4$
we obtain a non-trivial relation $h_1\Phi_1+h_2 \Phi_2=0$ which implies that 
$\Phi_1\wedge\Phi_2$ is zero, contradicting relation (6).
Hence any non-trivial relation corresponds to a $1$-dimensional subspace. 
But this implies that the relation is essentialy unique: if we had two such 
relations that are independent over $M$ we would obtain by the same argument 
$\Phi_1 \wedge \Phi_2=0$. 
One can then check that  (for $k\geq 1$) the dimension of
$\dim S_{1,3k+1}(\Gamma[\sqrt{-3}])$ indeed equals 
$3 \dim M_{0,3k-6}(\Gamma[\sqrt{-3]})-\dim M_{0,3k-9}(\Gamma[\sqrt{-3}])= k^2-1$ and the result follows.
\end{proof}

\end{subsection}
\begin{subsection}{The modules ${\mathcal M}_1^1$ and  $\Sigma_1^1$}

\begin{theorem}
The $M$-module ${\mathcal M}_1^1$ is generated by an $S_4$-anti-invariant 
Eisenstein series of weight $(1,1)$ and an $S_4$-invariant Eisenstein series
of weight $(1,4)$.
\end{theorem}

This results in the following table for the irreducible
representations contained in $M_{1,1+3k}(\Gamma[\sqrt{-3}],{\rm det})$.

\bigskip
\vbox{
\centerline{\def\quad{\hskip 0.3em\relax}
\vbox{\offinterlineskip
\hrule
\halign{&\vrule#& \quad \hfil#\hfil \strut \quad  \cr
height2pt&\omit&&\omit&&\omit&&\omit&&\omit&&\omit& \cr
& $k$ && $s[4]$ && $s[3,1]$ && $s[2,2]$ && $s[2,1,1]$ && s[1,1,1,1]& \cr
height2pt&\omit&&\omit&&\omit&&\omit&&\omit&&\omit& \cr
\noalign{\hrule}
height2pt&\omit&&\omit&&\omit&&\omit&&\omit&&\omit& \cr
& $0$ && $0$ && $0$ && $0$ && $0$ && {\color{green}$1$} & \cr
& $1$ && {\color{green}$1$} && {\color{green}$1$} && $0$ && $0$ && $0$ & \cr
& $2$ && $0$ && $0$ && {\color{blue}$1$} && ${\color{red}1}+{\color{green}1}$ && {\color{green}$1$} & \cr
& $3$ && ${\color{green}1}+{\color{blue}1}$ && ${\color{green}1}+2$ && {\color{blue}$1$} && {\color{red}$1$} && $0$ & \cr
& $4$ && $0$ && $2$ && $2$ && $4$ && $3$ & \cr
& $5$ && $3$ && $6$ && $3$ && $3$ && $0$ & \cr
height2pt&\omit&&\omit&&\omit&&\omit&&\omit&&\omit& \cr
} \hrule}
}}

\begin{proof}
Let $E_{1,1}=\Psi_1/\zeta$ and $E_{1,4}=\Psi_2/\zeta$  
be the two Eisenstein series. Since $\Psi_1\wedge \Psi_2$ does not vanish we 
cannot have relations of the form $\alpha E_{1,1}+\beta E_{1,4}=0$ with 
$\alpha$ and $\beta$ scalar-valued modular forms. Therefore these forms
generate a submodule with graded piece of degree $1+3k$ of dimension 
equal to $\dim M_{3k}^0+\dim M_{3k-3}^0= (k+1)^2$.
\end{proof}

From the structure of the module ${\mathcal M}_1^1$ one can deduce the
structure of the $M$-module $\Sigma_1^1$. Recall that we constructed forms
$a_i$ and $b_{ij,kl}$ in Section \ref{sec:det1-case}.
We summarize our results.
\begin{theorem}
There exists three modular forms $A_1, A_2, A_3$ spanning the $s[2,1,1]$-space 
of $S_{1,7}(\Gamma[\sqrt{-3}],\det{})$ and two modular forms $B_1, B_2$ spanning the $s[2,2]$-space of $S_{1,7}(\Gamma[\sqrt{-3}],\det{})$ that generate
$\Sigma_1^1$ over $M$ and the module of relations is generated over $M$ 
by three relations spanning a representation of type $s[2,1,1]$.
\end{theorem}

\end{subsection}
\begin{subsection}{The module ${\mathcal M}_1^2=\Sigma_1^2$}

\begin{theorem}
There exists an $S_4$-invariant modular cusp 
form $\Psi_1$ in $S_{1,7}(\Gamma[\sqrt{-3}],{\rm det}^2)$, 
and  a $S_4$-anti-invariant form $\Psi_2$ in
 $S_{1,10}(\Gamma[\sqrt{-3}],{\rm det}^2)$
that generate $\Sigma_1^2$ as a free $M$-module: 
for $k\geq 1$ we have
$$
S_{1,3k+4}(\Gamma[\sqrt{-3}],{\rm det}^2)
= M_{3k-3}\, \Psi_1 \oplus M_{3k-6} \Psi_2 \, .
$$
\end{theorem}

This results in the following table for the irreducible
representations contained in $S_{1,1+3k}(\Gamma[\sqrt{-3}],{\rm \det}^{2})$.

\bigskip
\vbox{
\centerline{\def\quad{\hskip 0.3em\relax}
\vbox{\offinterlineskip
\hrule
\halign{&\vrule#& \quad \hfil#\hfil \strut \quad  \cr
height2pt&\omit&&\omit&&\omit&&\omit&&\omit&&\omit& \cr
& $k$ && $s[4]$ && $s[3,1]$ && $s[2,2]$ && $s[2,1,1]$ && s[1,1,1,1]& \cr
height2pt&\omit&&\omit&&\omit&&\omit&&\omit&&\omit& \cr
\noalign{\hrule}
height2pt&\omit&&\omit&&\omit&&\omit&&\omit&&\omit& \cr
& $0$ && $0$ && $0$ && $0$ && $0$ && $0$ & \cr
& $1$ && $0$ && $0$ && $0$ && $0$ && $0$ & \cr
& $2$ && {\color{blue}$1$} && $0$ && $0$ && $0$ && $0$ & \cr
& $3$ && $0$ && $0$ && $0$ && {\color{red}$1$} && {\color{red}$1$} & \cr
& $4$ && $1$ && $2$ && $1$ && $0$ && $0$ & \cr
& $5$ && $0$ && $1$ && $1$ && $3$ && $2$ & \cr
& $6$ && $3$ && $4$ && $2$ && $2$ && $0$ & \cr
height2pt&\omit&&\omit&&\omit&&\omit&&\omit&&\omit& \cr
} \hrule}
}}

\begin{proof}
Note that by (8) we have $\Psi_1\wedge \Psi_2\neq 0$, so there are
no relations between $\Psi_1$ and $\Psi_2$. 
This together with the dimension formula proves
our claim. 
\end{proof}

\begin{remark}
The form $\Psi_1$ is a lift of
$$
f_{\pm}=\sum a(n)q^n= q \pm 6\sqrt{10}\,  q^2+232\,  q^4+260\, q^7 + \ldots 
$$
in $S_8(\Gamma_0(9))$ with $f_{+}+f_{-}$ lying in the so-called `plus space';
$\Psi_1$ has eigenvalues $\lambda_{\nu}= a(p)+ \nu^2\bar{\nu}^5$ for
$\nu$ with ${\rm N}(\nu)=p\equiv 1 \, (\bmod \, 3)$.
\end{remark}
\end{subsection}

\end{section}
\begin{section}{The structure of ${\mathcal M}_2$}

\begin{subsection}{The module ${\mathcal M}_2^0=\Sigma_2^0$}
We begin by constructing some modular forms.
\begin{lemma}\label{sym2div}
The form ${\rm Sym}^2(\Phi_0)/\varphi_1\varphi_2(\varphi_2-\varphi_1)$
lies in $S_{2,5}(\Gamma[\sqrt{-3}])$.
\end{lemma}
\begin{proof}
Since $[\varphi_1,\varphi_2]= (\vartheta_1\vartheta_2)^2 [\vartheta_1,\vartheta_2]$
we have
$$
-36 \pi^2 \, {\rm Sym}^2(\Phi_0)
=(\vartheta_1\vartheta_2)^4 {\rm Sym}^2([\vartheta_1,\vartheta_2])
$$
from which it is obvious that we can divide by $\varphi_1\varphi_2$.
The stabilizer in $S_3$ of the space spanned by 
$\Phi_0 \in S_{1,7}(\Gamma[\sqrt{-3}])$ 
is isomorphic 
to $S_3$ and the orbit of $\varphi_1$ is 
$\{ \varphi_1, -\varphi_2, \varphi_2-\varphi_1\}$;
we thus 
see that we can divide by $\varphi_2-\varphi_1$ and obtain a modular form in
$M_{2,5}(\Gamma[\sqrt{-3}])$. 

Using the explicit formulae of the functions $\vartheta_1$ and $\vartheta_2$, we
see that the Fourier-Jacobi expansion of the form
$$
-36 \pi^2 
{\rm Sym}^2(\Phi_0)/\varphi_2\varphi_1(\varphi_2-\varphi_1)=
\vartheta_2\vartheta_1{\rm Sym}^2([\vartheta_1,\vartheta_2])/
(\varphi_2-\varphi_1)
$$
at the cusp $(1:0:0)$ has no constant term and 
using the action of the group $S_4$ on the cusps,
we see the same at the other cusps.
\end{proof}
We now put
$$
D_0= 
9\sqrt{-3} \frac{{\rm Sym}^2(\Phi_0)}{\varphi_1\varphi_2(\varphi_1-\varphi_2)}
 \quad \in S_{2,5}(\Gamma[\sqrt{-3}])
$$
with Fourier-Jacobi expansion of the last component of $D_0$
$$
(Y-Z) w+9( Y^3-3Y^2Z+3YZ^2-Z^3) w^3+ 8(Y^4-7Y^3Z+7YZ^3-Z^4) w^4+\dots
$$
and by the action of $R_3$ we get forms $D_1$ and $D_2$ whose Fourier-Jacobi
expansion is obtained by substituting $(\rho^iY,\rho^{2i}Z)$ ($i=1,2$) 
for $(Y,Z)$. These are linearly independent and generate a $S_4$-representation
of type $s[3,1]$.

We have $\Phi_0\wedge\Phi_1=c \zeta^2 \varphi_2$,  
$\Phi_0\wedge\Phi_2=-c \zeta^2 \varphi_1$ and
$\Phi_1\wedge\Phi_2=c \zeta^2 \varphi_0$ and from this we get
$$
{\rm Sym}^2(\Phi_0)\wedge {\rm Sym}^2(\Phi_1)\wedge {\rm Sym}^2(\Phi_2)=
-c^3 \zeta^6 \, \varphi_0\varphi_1\varphi_2.
$$
We conclude:

\begin{lemma}\label{Dwedges}
We have the identity
$$
D_0 \wedge D_1 \wedge D_2= -\rho c^3 \, \zeta^3\, .
$$
\end{lemma}

We now have the structure of $\Sigma_2^0$:
\begin{theorem}
The $M$-module $\Sigma_2^{0}$ is generated by the three modular forms 
$D_0$, $D_1$ and $D_3$ of weight $(2,5)$ that generate a $S_4$-representation
of type $s[3,1]$: for $k\geq 0$ we have
$$
S_{2,3k+5}(\Gamma[\sqrt{-3}]) = M_{3k}(\Gamma[\sqrt{-3}])
\otimes  \langle D_0, D_1, D_2 \rangle
$$
\end{theorem}
\begin{proof}
Suppose that there is a relation $\sum_{i=1}^3 f_i D_i$ with $f_i$
modular forms of some weight. Then the wedge $D_0\wedge D_1 \wedge D_2$
would vanish identically and this contradicts Lemma \ref{Dwedges}.  
The dimension formula from \cite{B-vdG} we need is
$$
\dim S_{2,3k+2}(\Gamma[\sqrt{-3}])= 3 \, k\, (k+1)/2 \qquad
\hbox{\rm for $k\geq 1$}
$$
The dimension argument is the same as in the beginning of the preceding section.
It suffices to have the formula for $\dim S_{2,3k+2}(\Gamma_1[\sqrt{-3}])$.
This finishes the proof.
\end{proof}

We finish by giving a table for the irreducible representations 
contained in $S_{2,2+3k}(\Gamma[\sqrt{-3}])$.

\bigskip
\vbox{
\centerline{\def\quad{\hskip 0.3em\relax}
\vbox{\offinterlineskip
\hrule
\halign{&\vrule#& \quad \hfil#\hfil \strut \quad  \cr
height2pt&\omit&&\omit&&\omit&&\omit&&\omit&&\omit& \cr
& $k$ && $s[4]$ && $s[3,1]$ && $s[2,2]$ && $s[2,1,1]$ && s[1,1,1,1]& \cr
height2pt&\omit&&\omit&&\omit&&\omit&&\omit&&\omit& \cr
\noalign{\hrule}
height2pt&\omit&&\omit&&\omit&&\omit&&\omit&&\omit& \cr
& $0$ && $0$ && $0$ && $0$ && $0$ && $0$ & \cr
& $1$ && $0$ && {\color{red}$1$} && $0$ && $0$ && $0$ & \cr
& $2$ && $0$ && $1$ && $1$ && $1$ && $1$ & \cr
& $3$ && $1$ && $3$ && $1$ && $2$ && $0$ & \cr
height2pt&\omit&&\omit&&\omit&&\omit&&\omit&&\omit& \cr
} \hrule}
}}

\end{subsection}
\begin{subsection}{The module ${\mathcal M}_2^1=\Sigma_2^1$}
We begin by constructing the modular forms in 
$S_{2,5}(\Gamma[\sqrt{-3}],\det{})$
by considering
$$
D_0'=\frac{\zeta (D_1+D_2)}{\varphi_0(\varphi_1-\varphi_2)}
$$
and observing that these are holomorphic using the expansion along
the $C_{ij}$.
The action of $S_4$ thus gives rise to forms $D_i'$ for $i=0,1,2$ 
in $S_{2,5}(\Gamma[\sqrt{-3}],\det{})$.
The $D_i'$ generate a representation of type $s[3,1]$.

\begin{theorem} The $M$-module $\Sigma_2^1$ is generated by the three
modular forms $D_0'$, $D_1'$ and $D_2'$  of weight $(2,5)$ that generate
a $S_4$-representation of type $s[3,1]$: for $k\geq 0$ we have
$$
S_{2,3k+5}(\Gamma[\sqrt{-3}],\det)=M_{3k}(\Gamma[\sqrt{-3}])\otimes \langle
D_0',D_1',D_2'\rangle \, .
$$
\end{theorem}
\begin{proof}
The wedge $D_0'\wedge D_1'\wedge D_2'$ is a non-zero multiple of $\zeta^3$,
and this shows that there can be no relations between these generators.
The dimension formula now implies the result.
\end{proof}

The eigenvalues of Hecke eigenforms in $S_{2,5}(\Gamma[\sqrt{-3}],\det{})$ 
are the $F$-conjugates of the corresponding ones in 
$S_{2,5}(\Gamma[\sqrt{-3}])$ as follows from cohomological arguments,
cf.\ \cite{B-vdG}.
By the preceding two theorems the spaces $S_{2,2+3k}(\Gamma[\sqrt{-3}])$ and 
$S_{2,2+3k}(\Gamma[\sqrt{-3}],\det)$ 
are isomorphic as $S_4$-representations.
\end{subsection}
\begin{subsection}{The module ${\mathcal M}_2^2$}

\begin{theorem}
The $M$-module ${\mathcal M}_2^2$ is freely 
generated by an $S_4$-invariant form
$K_2$ of weight $(2,2)$, an $S_4$-anti-invariant form
$K_5$ of weight $(2,5)$ and an $S_4$-invariant form
$K_8$ of weight $(2,8)$:
$$
M_{2,2+3k}(\Gamma[\sqrt{-3}],{\rm det}^2)=
M_{0,3k}\,  K_2 \oplus M_{0,3k-3}\, K_5 \oplus M_{0,3k-6} \, K_8 \, .
$$
\end{theorem}

We give the irreducible representations in $M_{2,2+3k}(\Gamma[\sqrt{-3}],\det{}^2)$ for a few values of $k$.

\bigskip
\vbox{
\centerline{\def\quad{\hskip 0.3em\relax}
\vbox{\offinterlineskip
\hrule
\halign{&\vrule#& \quad \hfil#\hfil \strut \quad  \cr
height2pt&\omit&&\omit&&\omit&&\omit&&\omit&&\omit& \cr
& $k$ && $s[4]$ && $s[3,1]$ && $s[2,2]$ && $s[2,1,1]$ && s[1,1,1,1]& \cr
height2pt&\omit&&\omit&&\omit&&\omit&&\omit&&\omit& \cr
\noalign{\hrule}
height2pt&\omit&&\omit&&\omit&&\omit&&\omit&&\omit& \cr
& $0$ && ${\color{green}1}$ && $0$ && $0$ && $0$ && $0$ & \cr
& $1$ && $0$ && $0$ && $0$ && ${\color{green}1}$ && ${\color{green}1}$ & \cr
& $2$ && ${\color{green}1}+{\color{blue}1}$ && ${\color{green}1}+{\color{red}1}$ && ${\color{blue}1}$ && $0$ && $0$ & \cr
& $3$ && $0$ && $1$ && $1$ && $4$ && $2$ & \cr
& $4$ && $4$ && $5$ && $3$ && $2$ && $0$ & \cr
height2pt&\omit&&\omit&&\omit&&\omit&&\omit&&\omit& \cr
} \hrule}
}}

\begin{proof}
The form $K_2={\rm Sym}^2(E_{1,1})$; it has an alternative description as
$(\varphi_0D_0+\varphi_1D_1+\varphi_2D_2)/\zeta$. The form
$K_5$ is the form
$$
(\varphi_0(\varphi_0-\varphi_1-\varphi_2)\, D_0 
+(\varphi_1(-\varphi_0+\varphi_1-\varphi_2)\, D_1 +
\varphi_2 (-\varphi_0-\varphi_1+\varphi_2)\, D_2)/\zeta
$$
which is $S_4$-anti-invariant, while $K_8=
{\rm Sym}^2(E_{1,4})$. The wedge $K_2\wedge K_5\wedge K_8$ is the product of
$D_0\wedge D_1\wedge D_2$ with a non-zero rational function, 
hence by Lemma \ref{Dwedges} does not vanish. 
This implies that there can be no relations 
over $M$ between these generators. Therefore these elements generate
a $M$-module with graded piece of dimension $(3\, k^2+3\, k+2)/2$. 
By the
dimension formula (see Section \ref{modules}) these generators 
thus exhaust the whole module ${\mathcal M}_2^2$.

\end{proof}

\end{subsection}

\end{section}
\begin{section}{The structure of ${\mathcal M}_3^0$}
In this section we give the structure of the $M$-modules
${\mathcal M}_3^0$ and $\Sigma_3^0$. 
We start by constructing Eisenstein series in weight $(3,3)$. 
For this we define
$$
E_i= \frac{{\rm Sym}^3(\Phi_i)}{(\varphi_{i+1}\varphi_{i+2} 
(\varphi_{i+1}-\varphi_{i+2}))^2}
$$
for $i=0,1,2$ (taken $\bmod \, 3)$ and
$$
E_3=\frac{{\rm Sym}^3(\Phi_0+\Phi_1+\Phi_2)}{((\varphi_0-\varphi_1)
(\varphi_0-\varphi_2)(\varphi_1-\varphi_2))^2}.
$$
\begin{lemma}\label{Eis33}
The forms $E_0, E_1, E_2$ and $E_3$ are modular forms  of weight $(3,3)$
on $\Gamma[\sqrt{-3}]$ and at each of the four cusps exactly one of these
is non-zero. They generate a $S_4$-representation $s[2,1,1]\oplus s[1,1,1,1]$
and we have
$$
E_0\wedge E_1 \wedge E_2 \wedge E_3= c_2 \, \zeta^3\, \quad 
\hbox{\rm with $c_2 \in {\CC}^*$}
$$
\end{lemma}
\begin{proof}
Just as above in the proof of Lemma \ref{sym2div} 
we see that ${\rm Sym}^3(F_0)$ is divisible by
$\varphi_1^2\varphi_2^2$ and then using the action of $S_4$ one checks
that it is also divisible by $(\varphi_1-\varphi_2)^2$. The form
$ E_0+E_1+E_2-E_3$ is anti-invariant under $S_4$ and for $i=0,1,2$ 
the forms $4E_i+E_3-(E_0+E_1-E_2)$ generate a $S_4$-representation $s[2,1,1]$.
One calculates the Fourier-Jacobi expansion of these series.
For example one finds for the $4$th component of $E_0$ the expansion
$$
-3^{-5}\left((Y-Z)w+6(Y^2-Z^2)w^2+9(3Y^3+Y^2Z-YZ^2-3Z^3)w^3+\ldots \right)
$$
and the ones of $E_i$ ($i=1,2$) are obtained by substituting 
$(\rho^iY, \rho^{2i}Z)$ for $(Y,Z)$, while the fourth 
component of $E_3$ has an expansion
$$
-3^{-3}\left( (Y^3-Z^3)w^3- 18(Y^4Z-YZ^4)w^5+\ldots \right).
$$
One also checks that the coefficient of $w$ in the Fourier-Jacobi 
expansion of $E_0$
is
$$
-3^{-5} \left(\begin{matrix} (\kappa(Y'-Z')^3/(Y-Z)^2 \\ 
\kappa(Y'-Z')^2/(Y-Z) \\ \kappa (Y'-Z') \\
(Y-Z) \\ \end{matrix} \right)
$$
but that $E_3$ does not vanish at the first cusp.
From these facts the proof can be deduced.
\end{proof}
The $E_i$ are Hecke eigenforms with eigenvalues as given in
formulas (9a) and (9b).

\begin{proposition}
The $M$-module ${\mathcal M}_3^0$ is freely generated by the Eisenstein series
$E_i$ $i=0,1,2,3$.
\end{proposition}

We give the irreducible representations in $M_{3,3+3k}(\Gamma[\sqrt{-3}])$ 
for a few values of $k$.

\bigskip
\vbox{
\centerline{\def\quad{\hskip 0.3em\relax}
\vbox{\offinterlineskip
\hrule
\halign{&\vrule#& \quad \hfil#\hfil \strut \quad  \cr
height2pt&\omit&&\omit&&\omit&&\omit&&\omit&&\omit& \cr
& $k$ && $s[4]$ && $s[3,1]$ && $s[2,2]$ && $s[2,1,1]$ && s[1,1,1,1]& \cr
height2pt&\omit&&\omit&&\omit&&\omit&&\omit&&\omit& \cr
\noalign{\hrule}
height2pt&\omit&&\omit&&\omit&&\omit&&\omit&&\omit& \cr
& $-1$ && $0$ && $0$ && $0$ && $0$ && $0$ & \cr
& $0$ && $0$ && $0$ && $0$ && ${\color{green}1}$ && ${\color{green}1}$ & \cr
& $1$ && ${\color{green}1}$ && ${\color{green}1}+{\color{red}1}$ && $1$ && ${\color{red}1}$ && $0$ & \cr
& $2$ && $0$ && $2$ && $2$ && $4$ && $2$ & \cr
& $3$ && $3$ && $6$ && $3$ && $4$ && $1$ & \cr
height2pt&\omit&&\omit&&\omit&&\omit&&\omit&&\omit& \cr
} \hrule}
}}

\begin{proof}
By Lemma \ref{Eis33} there can be no relations between the $E_i$ 
with coefficients in the vector space $M_{k}$. 
Therefore these Eisenstein series generate
a $M$-module with as graded part of weight $3k+3$ 
a vector space $M_{3,3k+3}$ of dimension equal to $4\dim M_{3k}$.
The dimension formula says that the dimension of $M_{3,3k+3}$ equals
$2k^2+6k+4$. This shows that the $E_i$ generate the whole module.
\end{proof}
It is easy to check that the forms 
$$
\begin{aligned}
G_0=& \varphi_2E_1-\varphi_1E_2+ (\varphi_2-\varphi_1)E_3, \quad \\
G_1=& -\varphi_2E_0+\varphi_0E_2 +(\varphi_0-\varphi_2)E_3, \quad \\
G_2=& \varphi_1E_0-\varphi_0E_1+(\varphi_1-\varphi_0)E_3, \quad \\
\end{aligned}
$$
generate a $s[2,1,1]$-representation
in $S_{3,6}(\Gamma[\sqrt{-3}])$. Similarly the forms
$$
\begin{aligned}
H_1=& \varphi_1E_0+(\varphi_0-\varphi_2)E_1-\varphi_1E_2+(\varphi_0-\varphi_2)E_3\\
H_2=& \varphi_2E_0- \varphi_2E_1+(\varphi_0-\varphi_1)E_2+(\varphi_0-\varphi_1)E_3\\
\end{aligned}
$$
generate a representation of type $s[2,2]$ in 
$S_{3,6}(\Gamma[\sqrt{-3}])$.
Finally, the orbit of the form
$$
J_0=3\varphi_1E_0+(\varphi_0+\varphi_2)E_1+3\varphi_1E_2 
+(\varphi_0-2\varphi_1+\varphi_2)E_3
$$
generates a representation of type $s[3,1]$.
Writing the forms in terms of multiples of $E_i$ 
with $0\leq i \leq 3$ with coeffients from $M_6(\Gamma[\sqrt{-3}])$ 
it is obvious that we have the
invariant relation
$$
\varphi_0 G_0+\varphi_1 G_1 +\varphi_2 G_2 =0. \eqno(R4)
$$
We also have an $s[3,1]$-space of relations generated (under the $S_4$-action)
by the following relation between the elements $K_{02}=\varphi_0E_2$, 
$K_{12}=\varphi_1E_2$,
$K_{23}=(\varphi_1-\varphi_0)E_3$ and $K_{13}=(\varphi_2-\varphi_0)E_3$
of $S_{3,6}(\Gamma[\sqrt{-3}])$:
$$
\varphi_1K_{02}-\varphi_0K_{12}-(\varphi_2-\varphi_0)K_{23}+
(\varphi_1-\varphi_0)K_{13}=0 \eqno(R5)
$$
that generates a space of relations of type $s[3,1]$

\begin{theorem}
The $M$-module $\Sigma_3^0$ is generated by the $S_4$-orbits of the forms
$G_0$, $H_1$ and $J_0$. The relations are generated over $M$ by the $S_4$-orbits
of the relations $R4$ and $R5$. 
\end{theorem}
\begin{proof}
We know that the dimension of $M_{3,3k+3}=2k^2+6k+4$ for $k\geq 0$,
hence $\dim S_{3,3k+3}=2k^2+6k$. Any element of $S_{3,3k+3}$ can be written as
a linear combination $\sum_{i=0}^3 \epsilon_i E_i$ with $\epsilon_i \in M_{3k}$
a modular form that vanishes in the cusp where $E_i$ does not vanish. It is
easy to see that these are generated by the $G_i$, $H_i$ and the orbit of 
$J_0$.
\end{proof}
\end{section}
\begin{section}{Eigenvalues of Hecke operators}
In this section we shall give a number of eigenvalues for
vector-valued modular forms. A number of modular forms we encountered 
are lifts from $U(1)$ or ${\rm GL}(2)$. In general there are 
Kudla lifts (\cite{K1,K2})
$$
S_{b+2}(\Gamma_1(3^?)) \to S_{a,b+3}(\Gamma[\sqrt{-3}],\det{}^{\ell})
$$
with $\ell\equiv a (\bmod \, 3)$. Another type of lifts is given in \cite{Rog}
$$
S_{a+b+3}(\Gamma_1(3)) \to S_{a,b+3}(\Gamma[\sqrt{-3}],\det{}^{2})
$$
In the first case the eigenvalues for $T_{\nu}$ are $a(p) \nu^{a+1} +\bar{\nu}^{a+b+2}$ and in the second case they are of the form $a(p)+\nu^{a+1}\bar{\nu}^{b+1}$.
In all tables below the cusp form is not a lift and 
corresponds to a $3$-dimensional Galois 
representation.

\begin{subsection}{Eigenforms in $\Sigma_1^0$}

\begin{example}
We consider the eigenforms $\Phi_i \in S_{1,7}(\Gamma[\sqrt{-3}])$
with representation $s[2,1,1]$. We give a table of eigenvalues.

\begin{footnotesize}
\smallskip
\vbox{
\bigskip\centerline{\def\quad{\hskip 0.6em\relax}
\def\quod{\hskip 0.5em\relax }
\vbox{\offinterlineskip
\hrule
\halign{&\vrule#&\strut\quod\hfil#\quad\cr
height2pt&\omit&&\omit&&\omit&\cr
&$\alpha$ && $p$ && $\lambda_{\alpha}(\Phi_i)$ &\cr
\noalign{\hrule}
&$1+3\rho$&& $7$ && $759+261\rho$ & \cr
&$1-3\rho$&& $13$ && $-4137+1683\rho$ & \cr
& $-2+3\rho$ && $19$ && $24042+14733\rho$ & \cr
& $1+6\rho$ && $31$ && $-145401-241830\rho$ & \cr
& $4-3\rho$ && $37$ && $12900-114849\rho$  & \cr
& $1-6\rho$ && $43$ && $246567-8946\rho$  & \cr
& $4+9\rho$ && $61$ && $1048836-173205\rho$ & \cr
& $-2-9\rho$ && $67$ && $-1539510-1246887\rho$ & \cr
& $1+9\rho$ && $73$ && $-1563729+1261143\rho$  & \cr
& $7-3\rho$ && $79$ && $9921297+3294171\rho$  & \cr
& $-8+3\rho$ && $97$ && $5678616-3870891\rho$  & \cr
& $-2$ && $2$ && $72$ & \cr
& $-5$ && $5$ && $89622$ & \cr
} \hrule}
}}
\end{footnotesize}

\end{example}

\begin{example}
The space $S_{1,10}(\Gamma[\sqrt{-3}])$ has a decomposition
$s[3,1]\oplus s[2,2]\oplus s[2,1,1]$ as $S_4$-representation.
In the $s[3,1]$-space  and $s[2,2]$-space there are eigenforms
with eigenvalues

\begin{footnotesize}
\smallskip
\vbox{
\bigskip\centerline{\def\quad{\hskip 0.6em\relax}
\def\quod{\hskip 0.5em\relax }
\vbox{\offinterlineskip
\hrule
\halign{&\vrule#&\strut\quod\hfil#\quad\cr
height2pt&\omit&&\omit&&\omit&\cr
&$p$ && $s[3,1]$ && $s[2,2]$ &\cr
\noalign{\hrule}
& $7$ && $-13515+3573\rho$  && $15159+10863\rho$ & \cr
& $13$ && $-321963-290475\rho$ && $-95001+288351\rho$ & \cr
& $19$ && $2154864+1895139\rho$ && $-2977296-681147\rho$ & \cr
& $31$ && $-4371693-1547568\rho$ && $-24682119-22711896\rho$ & \cr
& $37$ && $-13227720-83952837\rho$ && $-76866504+46681047\rho$ & \cr
& $43$ && $108861123-48030912\rho$ && $32373957+31482576\rho$  & \cr
& $61$ && $1122962232+554059467\rho$ && $465758040+641801907\rho$  & \cr
& $67$ && $878127888+1196423595\rho$  && $-211962336+187424901\rho$ & \cr
& $73$ && $-1637757627-2807114427\rho$  && $3493044975+565725087\rho$  & \cr
& $79$ && $-504410811-607778811\rho$   && $-3018458193-2809124073\rho$ & \cr
& $97$ && $-23598528-7910853813\rho$   && $-6587510640-8420338791\rho$  & \cr
& $-2$  && $36$  && $1008$ & \cr
& $-5$  && $13464990$ && $-11930940$  & \cr
} \hrule}
}}
\end{footnotesize}

In the $s[2,1,1]$-part we find an eigenform
with eigenvalues

\begin{footnotesize}
\smallskip
\vbox{
\bigskip\centerline{\def\quad{\hskip 0.6em\relax}
\def\quod{\hskip 0.5em\relax }
\vbox{\offinterlineskip
\hrule
\halign{&\vrule#&\strut\quod\hfil#\quad\cr
height2pt&\omit&&\omit&&\omit&\cr
&$\alpha$ && $p$ && $s[2,1,1]$ &\cr
\noalign{\hrule}
&$1+3\rho$&& $7$ && $26985+20097\rho$ & \cr
&$1-3\rho$&& $13$ && $31521+13761\rho$ & \cr
& $-2+3\rho$ && $19$ && $1806888+842463\rho$ & \cr
& $1+6\rho$ && $31$ && $15911679+12552264\rho$ & \cr
& $4-3\rho$ && $37$ && $81911640+71598267\rho$  & \cr
& $1-6\rho$ && $43$ && $47737551-26870472\rho$  & \cr
& $4+9\rho$ && $61$ && $524111736+375028731\rho$ & \cr
& $-2-9\rho$ && $67$ && $489305208-4044033\rho$ & \cr
& $1+9\rho$ && $73$ && $-513904983+1961971497\rho$  & \cr
& $7-3\rho$ && $79$ && $-203501319-3483886959\rho$  & \cr
& $-8+3\rho$ && $97$ && $4237830912-1749247641\rho$  & \cr
& $-2$ && $2$ && $-1548$ & \cr
& $-5$ && $5$ && $-1356390$ & \cr
} \hrule}
}}
\end{footnotesize}

\end{example}

\end{subsection}

\begin{subsection}{Eigenforms in $\Sigma_1^1$}

\begin{example}
The space $S_{1,7}(\Gamma[\sqrt{-3}],\det{})$ equals $s[2,2]\oplus s[2,1,1]$
as a representation of $S_4$. An eigenform $F$ in the $s[2,1,1]$-space  is a
lift, so
for a prime $p\equiv  1 (\bmod \, 3)$ and $\nu \in O_F$ with
$\nu\equiv 1 (\bmod \, 3)$ of norm $p$, the eigenvalue
$\lambda_{\nu}$ satisfies 
$\lambda_{\nu}= a_p \, \nu^2+\bar{\nu}^7$
with $a_p$ the eigenvalue of $(\eta(3\tau)\eta(\tau))^6 \in S_6(\Gamma_0(3))$.
An eigenform in the $s[2,2]$-space 
has eigenvalues

\begin{footnotesize}
\smallskip
\vbox{
\bigskip\centerline{\def\quad{\hskip 0.6em\relax}
\def\quod{\hskip 0.5em\relax }
\vbox{\offinterlineskip
\hrule
\halign{&\vrule#&\strut\quod\hfil#\quad\cr
height2pt&\omit&&\omit&&\omit&\cr
&$\alpha$ && $p$ && $\lambda_{\alpha}(F)$ &\cr
\noalign{\hrule}
&$1+3\rho$&& $7$ && $-294+855\rho$ & \cr
&$1-3\rho$&& $13$ && $-2220-1017\rho$ & \cr
& $-2+3\rho$ && $19$ && $5817-5841\rho$ & \cr
& $1+6\rho$ && $31$ && $-23847-38466\rho$ & \cr
& $4-3\rho$ && $37$ && $152301-21375\rho$  & \cr
& $1-6\rho$ && $43$ && $-188403-18558\rho$  & \cr
& $-2$ && $2$ && $180$ & \cr
& $-5$ && $5$ && $82764$ & \cr
} \hrule}
}}
\end{footnotesize}

\end{example}

\begin{example}
The space $S_{1,10}(\Gamma[\sqrt{-3},\det{})$ decomposes as 
$s[4]\oplus 2\, s[3,1]\oplus s[2,2]\oplus s[2,1,1]$. 
The $S_4$-invariant eigenform 
is a lift of a form $g\in S_9(\Gamma_1(9))$ with $g=q+45\, q^3-284 \, q^4+1512 \, q^6 +\ldots$. The eigenform in the $s[2,2]$-space is a lift
of a form $g \in S_9(\Gamma_0(9),\psi)$ with $q$-expansion
$q+238\, q^4+1652 \, q^7 -4194 \, q^{10}+\ldots $. 
In the $s[2,1,1]$ we find an eigenform with eigenvalues as in the following table

\begin{footnotesize}
\smallskip
\vbox{
\bigskip\centerline{\def\quad{\hskip 0.6em\relax}
\def\quod{\hskip 0.5em\relax }
\vbox{\offinterlineskip
\hrule
\halign{&\vrule#&\strut\quod\hfil#\quad\cr
height2pt&\omit&&\omit&&\omit&\cr
&$\alpha$ && $p$ && $\lambda_{\alpha}(F_1)$ &\cr
\noalign{\hrule}
&$1+3\rho$&& $7$ && $-19320-7497\rho$ & \cr
&$1-3\rho$&& $13$ && $74208-298521\rho$ & \cr
& $-2+3\rho$ && $19$ && $877737+798561\rho$ & \cr
& $1+6\rho$ && $31$ && $10127631+22554360\rho$ & \cr
& $4-3\rho$ && $37$ && $-80206539-23638131\rho$  & \cr
& $1-6\rho$ && $43$ && $-113882937+15496200\rho$  & \cr
& $-2$ && $2$ && $-36$ & \cr
& $-5$ && $5$ && $1289610$ & \cr
} \hrule}
}}
\end{footnotesize}

\end{example}

\end{subsection}

\begin{subsection}{Eigenforms in $\Sigma_1^2$}
\begin{example}
The space $S_{1,7}(\Gamma[\sqrt{-3}],\det{}^2)$ is $1$-dimensional
and $S_4$-invariant. The generator  $\Psi_1$ is a lift 
of an element in $S_8(\Gamma_0(9))$.
The space $S_{1,10}(\Gamma(\sqrt{-3}),{\rm det}^2)$ decomposes as
$s[1,1,1,1]\oplus s[2,1,1]$,
where the component $s[1,1,1,1]$ is generated by the form $\Psi_2$ and the
component $s[2,1,1]$ is generated by the forms $\varphi_i\Psi_1$. 
All forms in $S_{1,10}(\Gamma(\sqrt{-3}),{\rm det}^2)$
have the same eigenvalues given in the next table.

\begin{footnotesize}
\smallskip
\vbox{
\bigskip\centerline{\def\quad{\hskip 0.6em\relax}
\def\quod{\hskip 0.5em\relax }
\vbox{\offinterlineskip
\hrule
\halign{&\vrule#&\strut\quod\hfil#\quad\cr
height2pt&\omit&&\omit&&\omit&\cr
&$\nu$ && $p$ && $\lambda_{\nu}(\Psi_2)$  & \cr
\noalign{\hrule}
&$1+3\rho$&& $7$ && $-6549-17352\rho$   & \cr
&$1-3\rho$&& $13$ && $223599+133992\rho$  &\cr
& $-2+3\rho$ && $19$ && $-492621-1294560\rho$  & \cr
& $1+6\rho$ && $31$ && $3832419+8618040\rho$ & \cr
& $4-3\rho$ && $37$ && $56905563+49705992\rho$  & \cr
& $1-6\rho$ && $43$ && $16590459+186818112\rho$  & \cr
& $-2$ && $2$ && $-684$ & \cr
& $-5$ && $5$ && $6541650$ & \cr
} \hrule}
}}
\end{footnotesize}

\end{example}
\end{subsection}

\begin{subsection}{Eigenforms in $\Sigma_2^0$}

\begin{example}
We give the Hecke eigenvalues for the $D_i$ of weight $(2,5)$ and of
the forms in the $s[2,1,1]$-part of $S_{2,8}(\Gamma[\sqrt{-3}])$.

\begin{footnotesize}
\smallskip
\vbox{
\bigskip\centerline{\def\quad{\hskip 0.6em\relax}
\def\quod{\hskip 0.5em\relax }
\vbox{\offinterlineskip
\hrule
\halign{&\vrule#&\strut\quod\hfil#\quad\cr
height2pt&\omit&&\omit&&\omit&&\omit&\cr
&$\nu$ && $p$ && weight (2,5), irrep s[3,1] && weight (2,8), irrep s[2,1,1] & \cr
\noalign{\hrule}
&$1+3\rho$&& $7$ && $-105-297\rho$ && $-3039-765\rho$  & \cr
&$1-3\rho$&& $13$ && $1137+945\rho$ && $97707+110007\rho$ &\cr
& $-2+3\rho$ && $19$ && $-1536-891\rho$ && $-268962-412137\rho$ & \cr
& $1+6\rho$ && $31$  && $20577-1728\rho$ && $4182969+2591334\rho$  & \cr
& $4-3\rho$ && $37$ && $19200-37017\rho$  && $-61836+6730299\rho$  & \cr
& $1-6\rho$ && $43$  && $-113667-127872\rho$ && $13604205-6584742\rho$ & \cr
& $4+9\rho$ && $61$  && $-354048-242433\rho$ && $-67731468-1452033\rho$  & \cr
& $-2-9\rho$ && $67$  && $271488+194805\rho$ && $-45800610+117273771\rho$ & \cr
& $1+9\rho$ && $73$ && $-268107-235467\rho$ && $-373673625-459690417\rho$ & \cr
& $7-3\rho$ && $79$ && $114159+449199\rho$ && $235630047+382294197\rho$ & \cr
& $-8+3\rho$ && $97$  && $-60288-554013\rho$ && $-95419824162-26086979421\rho$ & \cr
& $-2$ && $2$ && $-72$ && $-288$ & \cr
& $-5$ && $5$  && $-810$ && $1629990$ & \cr
} \hrule}
}}
\end{footnotesize}

\end{example}
\end{subsection}

\begin{subsection}{Eigenforms in $\Sigma_2^2$}
\begin{example}
The space $S_{2,8}(\Gamma[\sqrt{-3}],{\rm det}^2)$ splits as 
$s[4]\oplus s[3,1]\oplus s[2,2]$.
The $S_4$-invariant form is a lift of $U(1)$ with eigenvalues 
$\nu^9+\nu^3\bar{\nu}^6+\bar{\nu}^9$. The $s[2,2]$-forms are lifts 
from $S_7(\Gamma_1(3))$.
We give the eigenvalues for the $s[3,1]$-part of 
$S_{2,8}(\Gamma[\sqrt{-3}],{\rm det}^2)$.

\begin{footnotesize}
\smallskip
\vbox{
\bigskip\centerline{\def\quad{\hskip 0.6em\relax}
\def\quod{\hskip 0.5em\relax }
\vbox{\offinterlineskip
\hrule
\halign{&\vrule#&\strut\quod\hfil#\quad\cr
height2pt&\omit&&\omit&&\omit&\cr
&$\nu$ && $p$ &&  weight (2,8), irrep s[3,1] & \cr
\noalign{\hrule}
&$1+3\rho$&& $7$ && $-2175-1602\rho$   & \cr
&$1-3\rho$&& $13$ && $-58947-169740\rho$  &\cr
& $-2+3\rho$ && $19$ && $737949-220734\rho$  & \cr
& $1+6\rho$ && $31$  && $-90267+2362374\rho$   & \cr
& $4-3\rho$ && $37$ && $-8035881-17655156\rho$    & \cr
& $1-6\rho$ && $43$  && $6838329-67590\rho$  & \cr
& $-2$ && $2$ && $792$ & \cr
& $-5$ && $5$  && $-408510$ & \cr
} \hrule}
}}
\end{footnotesize}

\end{example}
\end{subsection}

\begin{subsection}{Eigenforms in $\Sigma_3^0$}

\begin{example}
We give the eigenvalues for the $s[2,1,1]$-part and the $s[3,1]$-part of
$S_{3,6}(\Gamma[\sqrt{-3}])$.

\begin{footnotesize}
\smallskip
\vbox{
\bigskip\centerline{\def\quad{\hskip 0.6em\relax}
\def\quod{\hskip 0.5em\relax }
\vbox{\offinterlineskip
\hrule
\halign{&\vrule#&\strut\quod\hfil#\quad\cr
height2pt&\omit&&\omit&&\omit&&\omit&  \cr
& $\alpha$ && $p$ && $s[2,1,1]$ && $s[3,1]$  &  \cr
\noalign{\hrule}
& $1+3\rho$&& $7$ && $3189-459\rho$  && $273+2457 \rho$ & \cr
& $1-3\rho$&& $13$ && $ -3543+8721\rho$ && $-35619-46683\rho$ & \cr
& $-2+3\rho$ && $19$ && $29784+118179\rho$ && $152256-30537\rho$  & \cr
& $1+6\rho$ && $31$ && $-29949-203904\rho$ && $-167001-547992\rho$ & \cr
& $4-3\rho$ && $37$ && $355296+8667\rho$  && $1545024+1338363\rho$ & \cr
& $1-6\rho$ && $43$ && $-66741-241272\rho$  && $2292303+207792\rho$ & \cr
& $4+9\rho$ && $61$ && $-99168+3835107\rho$ && $3969904-118989\rho$ & \cr
& $-2-9\rho$ && $67$ && $-5321544+13554459\rho$ && $11562096+21366423\rho$ & \cr
& $1+9\rho$ && $73$ && $-58317351-62040087\rho$ && $9680853+35351397\rho$ & \cr
& $7-3\rho$ && $79$ && $-44663451-34446411\rho$ && $1196481+1434969\rho$ & \cr
& $-8+3\rho$ && $97$ && $-52988496+9813663\rho$ && $2112240+65593827\rho$ & \cr
& $-2$ && $2$ && $-36$ && $-36$ & \cr
& $-5$ && $5$ && $-563670$ && $-117522$ & \cr
} \hrule}
}}
\end{footnotesize}

\end{example}
\end{subsection}

\end{section}

\end{document}